\documentclass[oneside]{amsart}
\usepackage[a4paper,total={120mm,240mm},top=50mm,bottom=50mm]{geometry}
\usepackage{amsthm}
\usepackage{amsthm} 
\theoremstyle{plain}

\newtheorem{theorem}{Theorem}[section]
\newtheorem{lemma}{Lemma}[section] 
\newtheorem{corollary}{Corollary}[section] 
\newtheorem{remark}{Remark}[section] 
\newtheorem{proposition}[theorem]{Proposition}

\usepackage{amsmath}
\usepackage{amsfonts}
\usepackage{graphics}
\usepackage{amssymb}
\usepackage{verbatim}
\usepackage{amscd}
\usepackage{enumerate}
\usepackage{graphicx}

\usepackage{hyperref}

\begin{document}

\title {On smoothing properties of the Bergman projection}

\keywords{Smoothing properties \; Bergman projection \; Hyperconvexity index \; Weighted estimates}

\email{ptrongthuc@hcmut.edu.vn}
  
\author[Phung Trong Thuc]{Phung Trong Thuc}

\address
	{Ho Chi Minh City University of Technology, VNU-HCM, Vietnam}	
 \thanks{This work was supported by the VNU-HCM program, Vietnam.}	

\subjclass[2010]{Primary 32A25; Secondary 32A36.}

\begin{abstract}
We study smoothing properties of the Bergman projection and also of weighted Bergman projections. In particular, we relate these properties to the hyperconvexity index of a pseudoconvex domain in $\mathbb{C}^{n}$. The notion of  a hyperconvexity index was first introduced by B.Y. Chen, which provides a flexible criterion for studying geometric properties of hyperconvex domains. We also obtain a new estimate of weighted Bergman projections, which improves a well-known estimate of Berndtsson and Charpentier. We give several applications of this estimate, including the study of smoothing properties of weighted Bergman projections.
\end{abstract}

\maketitle

\section{Introduction}
Let $\Omega\,\subset\,\mathbb{C}^{n}$ be a bounded pseudoconvex domain, and let $P_{\Omega}$ be the Bergman projection of $\Omega$, that is the orthogonal projection of $L^{2}\left(\Omega\right)$ onto
$A^{2}\left(\Omega\right)$. Here, for $1\leq p\leq\infty$, $A^{p}\left(\Omega\right)=\mathcal{{O}}\left(\Omega\right)\cap L^{p}\left(\Omega\right)$
denotes the space of holomorphic, $p$-integrable functions (always with respect
to the Lebesgue measure in $\mathbb{C}^{n}$).  And, for a positive measurable function $\omega$ on $\Omega$, we
denote by $L^{2}\left(\Omega,\omega\right)$ the space of all measurable
functions $f:\Omega\rightarrow\mathbb{C}$ satisfying 
\[
\int_{\Omega}\left|f\right|^{2}\omega\,<\,\infty.
\]

Recently, it has been shown that the Bergman projection has nice smoothing properties
when acting on the space of conjugate holomorphic functions. Here, we call a function $f$ a conjugate holomorphic function if its complex conjugate, $\overline{f}$, is holomorphic.  In \cite{HMS14},
the authors proved that for any $k\in\mathbb{N}$ and under the assumption of \emph{Condition R},
$\left\Vert P_{\Omega}\left(f\right)\right\Vert _{H^{k}\left(\Omega\right)}\leq C_{k}\left\Vert f\right\Vert _{L^{2}\left(\Omega\right)}$,
for all conjugate holomorphic functions $f\in L^{2}\left(\Omega\right)$. Here $H^{k}\left(\Omega\right)$ is the standard $L^{2}$ Sobolev
space of order $k$. In fact it is proved in \cite[Theorem 1.10]{HMS14} that, for any
smoothly bounded domain $\Omega$, by assuming that $P_{\Omega}$
is bounded from $H^{k_{1}}\left(\Omega\right)$ to $H^{k_{2}}\left(\Omega\right)$,
for some $k_{1},k_{2}\in\mathbb{N}$, then for every $g\in C^{\infty}\left(\overline{\Omega}\right)$,
\begin{equation}\label{herg}
\left\Vert P_{\Omega}\left(fg\right)\right\Vert _{H^{k_{2}}\left(\Omega\right)}\leq\text{const}\left\Vert f\right\Vert _{L^{2}\left(\Omega\right)},
\end{equation}
for all conjugate holomorphic functions $f$ in $L^{2}\left(\Omega\right)$. Recently, Herbig \cite{Herb13} showed  under the same assumption that one can weaken the right
hand side of \eqref{herg} to only $\left\Vert f\right\Vert _{H^{-k}\left(\Omega\right)}$,
for any $f\in L^2\left(\Omega\right)$ and  every given $k\in\mathbb{N}$.

It would be interesting to relate similar smoothing properties of the Bergman projection to geometric behavior of the domain, and also
to study Sobolev-norm estimates $\left\Vert \cdot\right\Vert _{H^{s}\left(\Omega\right)}$
, when the exponent $s$ is small (less than $1$). 
It is known that (see \cite[Theorem 2.4]{BeCh00}) for any pseudoconvex
domain $\Omega$ with Lipschitz boundary, $P_{\Omega}:H^{s}\left(\Omega\right)\rightarrow H^{s}\left(\Omega\right)$
is bounded, for any $0<s<\left.\eta\left(\Omega\right)\right/2$, where $\eta\left(\Omega\right)$ is the Diederich-Forn{\ae}ss index of $\Omega$,
defined by
\begin{align*}
\eta\left(\Omega\right) & =\sup\Bigl\{\eta>0:\text{there exist }\rho\in PSH^{-}\left(\Omega\right)\text{ and }C>0\\
 &\hspace*{2.4cm} \text{ such that  }\frac{1}{C}\delta_{\Omega}^{\eta}\leq-\rho\leq C\delta_{\Omega}^{\eta}\;\text{ on }\Omega\Bigr\}.
\end{align*}
Here we denote by $\delta_{\Omega}$ the boundary distance function,
and by $PSH^{-}\left(\Omega\right)$ the set of all negative plurisubharmonic
functions on $\Omega$. The Diederich-Forn{\ae}ss index is always positive for pseudoconvex
domains with Lipschitz boundary, see \cite{Har08}. It can be arbitrarily
close to $0$, as the smoothly bounded worm domain shown, see \cite{DiFo77}. On the other side, for any number $\mu>0$, Barrett \cite{Bar92} proved
that $P_{\Omega_{\mu}}:H^{s}\left(\Omega_{\mu}\right)\not\rightarrow H^{s}\left(\Omega_{\mu}\right)$,
for any $s\geq\left.\pi\right/2\mu$. Here $\Omega_{\mu}$ is the smooth worm domain with parameter $\mu$, see \cite[Definition 1.2]{KrPe08}.

In connection with the Diederich-Forn{\ae}ss index, in \cite[Theorem 1.4]{ChZe17},
the authors showed that for any smoothly bounded pseudoconvex domain $\Omega$ in $\mathbb{C}^{n}$ and any number $s<\left.\eta\left(\Omega\right)\right/\left(4n\right)$,
\[
\left\Vert P_{\Omega}\left(f\right)\right\Vert _{H^{s}\left(\Omega\right)}\leq\text{const}\left\Vert f\right\Vert _{L^{2}\left(\Omega\right)},
\]
for all conjugate holomorphic functions $f\in L^{2}\left(\Omega\right)$. As a corollary, the operator $f\rightarrow P_{\Omega}\left(\overline{f}\right)$ is compact in $A^{2}\left(\Omega\right)$.

Our first result in this paper is the following smoothing property of the Bergman projection, which, in particular, improves the exponent mentioned above in
the paper \cite{ChZe17}. 

\begin{theorem}\label{thm1}
Let $\Omega$ be a smoothly bounded pseudoconvex domain in $\mathbb{C}^{n}$. Assume
that there are $\eta>0$ and $\rho\in PSH^{-}\left(\Omega\right)$ 
such that $-\rho\leq C\delta_{\Omega}^{\eta}$ on $\Omega$, for some
positive constant $C$. Then for every function $g\in C^{\infty}\left(\overline{\Omega}\right)$, any $k\in\mathbb{Z}^{+}$, and any $0< t<\eta$, there is a positive constant $C_{\Omega,g,k,t}$
such that 
\begin{equation}\label{main:1}
\left\Vert P_{\Omega}\left(fg\right)\right\Vert _{H^{\left.t\right/2}\left(\Omega\right)}\leq C_{\Omega,g,k,t}\intop_{\Omega}\left|f\right|\delta_{\Omega}^{k},
\end{equation}
for all conjugate holomorphic functions $f\in L^{2}\left(\Omega\right)$.

\end{theorem}

The supremum $\alpha\left(\Omega\right)$ of all exponents $\eta$ stated
in Theorem \ref{thm1} is called the \emph{hyperconvexity index}
of $\Omega$ (see \cite{Che17}), that is
\begin{align*}
\alpha\left(\Omega\right) & =\sup\Bigl\{\eta>0:\text{ there exist }\rho\in PSH^{-}\left(\Omega\right)\text{ and }\,C>0\\
 & \hspace*{3.65cm}\text{ such that }-\rho\leq C\delta_{\Omega}^{\eta}\;\text{ on }\Omega\Bigr\}.
\end{align*}
If no such function $\rho$ exists then $\alpha\left(\Omega\right):=0$.

As noted by Chen \cite{Che17}, it is easier to verify the hyperconvexity
index than to justify the Diederich-Forn{\ae}ss index of a pseudoconvex
domain. It is also clear that the Diederich-Forn{\ae}ss index is
less than or equal to the hyperconvexity index, that is $\eta\left(\Omega\right)\leq\alpha\left(\Omega\right)$. From Theorem \ref{thm1}, in particular, we get that $P_{\Omega}\left(fg\right)\in H^{\eta}\left(\Omega\right)$,
for any $\eta<\left.\alpha\left(\Omega\right)\right/2$, and the inequality \eqref{main:1} holds for any $t<\alpha\left(\Omega\right)$. For the case $g\equiv1$, our result thus gives the following corollary:

\begin{corollary}
Let $\Omega$ be a smoothly bounded pseudoconvex domain in $\mathbb{C}^{n}$.
For any $k\in\mathbb{Z}^{+}$ and any $s<\left.\alpha\left(\Omega\right)\right/2$,
there is a positive constant $C$ such that
\[
\left\Vert P_{\Omega}\left(f\right)\right\Vert _{H^{s}\left(\Omega\right)}\leq C\left\Vert f\right\Vert _{H^{-k}\left(\Omega\right)},
\]
for all conjugate holomorphic functions $f$ in $L^{2}\left(\Omega\right)$.

\end{corollary}

\begin{remark}\normalfont
By extending the domain of definition of $P_{\Omega}$, we can even
weaken the assumption of $f\in L^{2}\left(\Omega\right)$ in Theorem
\ref{thm1} to only $f\in H^{\gamma}\left(\Omega\right)$, for any number
$\gamma>-\left.\alpha\left(\Omega\right)\right/2$. For the sake of discussion,
we leave it to Proposition \ref{impro}. In comparison to the estimate in \cite{ChZe17}, our result not only
improves the range in the Sobolev exponent, but in fact shows that
the Bergman projection has essentially $\alpha\left(\Omega\right)$-gain
in the Sobolev scale, namely $H^{s_{1}}\left(\Omega\right)\cap\overline{\mathcal{O}\left(\Omega\right)}\xrightarrow{P_{\Omega}}H^{s_{2}}\left(\Omega\right)$,
for any $\left.-\alpha\left(\Omega\right)\right/2<s_{1}<s_{2}<\left.\alpha\left(\Omega\right)\right/2$. In view of Barrett's result, this smoothing property provides a new
look in the Sobolev regularity of the Bergman projection with small Sobolev
exponents.

\end{remark}

Hyperconvexity indices can be used to study the \emph{integrability index} $\beta\left(\Omega\right)$
of the Bergman kernel $K_{\Omega}$, which is defined by
\[
\beta\left(\Omega\right)=\sup\left\{ \beta\geq2:K_{\Omega}\left(\cdot,w\right)\in L^{\beta}\left(\Omega\right)\text{ for all }w\in\Omega\right\} .
\]
For example, one can check that $\beta\left(\mathbb{H}\right)=4$, where $$\mathbb{H}=\left\{ \left(z_{1},z_{2}\right)\in\mathbb{C}^{2}:\left|z_{1}\right|<\left|z_{2}\right|<1\right\}$$
is the Hartogs triangle. For a general pseudoconvex domain $\Omega$,
$\beta\left(\Omega\right)$ might be arbitrarily close to $2$ as
the work of Barrett \cite{Bar92} shown. In \cite{Che17}, Chen proved
that for any pseudoconvex domain $\Omega$ in $\mathbb{C}^{n}$,
\begin{equation}\label{Chen1}
\beta\left(\Omega\right)\geq2+\dfrac{2\alpha\left(\Omega\right)}{2n-\alpha\left(\Omega\right)}.
\end{equation}

In this paper we would like to supplement this result with an estimate
in the case of pseudoconvex domains with $C^{2}$ boundary.

\begin{theorem}\label{thm_int}
Let $\Omega$ be a bounded pseudoconvex domain with $C^{2}$ boundary
in $\mathbb{C}^{n}$. Then
\begin{equation}\label{improChen}
\beta\left(\Omega\right)\geq2+\dfrac{2\alpha\left(\Omega\right)}{n+1-\alpha\left(\Omega\right)}.
\end{equation}

\end{theorem}

We would also like to study a similar smoothing property as in Theorem \ref{thm1} for weighted
Bergman projections. It is known that for any pseudoconvex domain
with smooth boundary, there exist $0<\eta<1$ and $\Psi\in C^{\infty}\left(\overline{\Omega}\right)$
such that $\Psi>0$ on $\overline{\Omega}$ and $\rho=-\delta^{\eta}\Psi$ is plurisubharmonic on $\Omega$.
Here $\delta\in C^{\infty}\left(\overline{\Omega}\right)$ is any function
that equals $\delta_{\Omega}$ near $\partial\Omega$. In fact, in
\cite{Ran81} Range showed that for pseudoconvex domains with $C^{3}$
boundary, $\rho$ can be chosen in the form
\[
\rho=-\delta^{\eta}e^{-K\psi},
\]
for some positive constant $K$, where $\psi\in C^{2}\left(\overline{\Omega}\right)$
is any strictly plurisubharmonic function on $\overline{\Omega}$.
Our next result is stated as follows:

\begin{theorem} \label{thm2}
Let $\Omega$ be a smoothly bounded pseudoconvex domain in $\mathbb{C}^{n}$. Let
$\delta\in C^{\infty}\left(\overline{\Omega}\right)$ be a function
that equals $\delta_{\Omega}$ near $\partial\Omega$. Let $\rho\in PSH^{-}\left(\Omega\right)$
be a function in the form $\rho=-\delta^{\eta}\Psi$, where $0<\eta<1$, and  $\Psi\in C^{\infty}\left(\overline{\Omega}\right)$
 such that $\Psi>0$ on $\overline{\Omega}$. Let $\alpha\in\left(\left.-1\right/4,\infty\right)$ and let $s_{1},s_{2}$
be numbers such that
\[
-\left(\frac{1+\sqrt{1+4\alpha}}{2}+\alpha\right)\frac{\eta}{2}<s_{1}<s_{2}<\left(\frac{1+\sqrt{1+4\alpha}}{2}-\alpha\right)\frac{\eta}{2}.
\]
Then for every $g\in C^{\infty}\left(\overline{\Omega}\right)$ and
any $k\in\mathbb{Z}^{+}$, there is a positive constant $C$ such
that
\[
\left\Vert P_{\left(-\rho\right)^{\alpha}}\left(fg\right)\right\Vert _{H^{s_{2}}\left(\Omega\right)}\leq C\intop_{\Omega}\left|f\right|\delta_{\Omega}^{k},
\]
for all conjugate holomorphic functions $f$ in $H^{s_{1}}\left(\Omega\right)$.
Here $P_{\left(-\rho\right)^{\alpha}}$ denotes the weighted Bergman
projection of $L^{2}\left(\Omega,\left(-\rho\right)^{\alpha}\right)$.

\end{theorem}

\begin{remark}\normalfont 
For pseudoconvex domains of finite type, in  \cite{CDM15} the class
of weights $\delta^{r}\Psi$, where $0\leq r$ and $\Psi>0$, $\Psi\in C^{\infty}\left(\overline{\Omega}\right)$,
has been studied.

When $\alpha=0$, the result essentially represents the smoothing phenomenon as in Theorem \ref{thm1} and Proposition \ref{impro}. However, in Theorem \ref{thm2} we impose
more conditions on $\rho$.

As $\alpha$ becomes larger, $P_{\left(-\rho\right)^{\alpha}}$ maps
into a Sobolev space with lower regularity. Theorem  \ref{thm2}, in
particular, says that, when acting on conjugate holomorphic functions, $P_{\left(-\rho\right)^{\alpha}}$ still maps
into a Sobolev space of a positive order (from a Sobolev space of
a negative order) as long as $\alpha<2$. 

It is also interesting to observe that $P_{\left(-\rho\right)^{\alpha}}$
has (essentially) $\left(\frac{1+\sqrt{1+4\alpha}}{2}\right)\eta$-gain
in the Sobolev exponent when acting on the space of conjugate holomorphic
functions. 

\end{remark}

One of the tools used in the proof of Theorem \ref{thm2}  is the following
estimate of weighted Bergman projections, which improves a result of  Berndtsson and Charpentier \cite[Theorem 2.1]{BeCh00}.

\begin{theorem} \label{BeCh1}
Let $\Omega$ be a bounded pseudoconvex domain in $\mathbb{C}^{n}$.
Let $\psi,\varphi\in C^{2}\left(\Omega\right)$ be such that $\psi + \varphi$ is plurisubharmonic
 on $\Omega$. Let $H\in L_{loc}^{\infty}\left(\Omega\right)$
be a non-negative function satisfying $\sup_{K}H<1$, for any subset
$K\Subset\Omega$. Assume also that
\begin{equation*}
i\overline{\partial}\psi\wedge\partial\psi\leq Hi\partial\overline{\partial}\left(\varphi+\psi\right)\text{ on }\Omega.
\end{equation*}
Then 
\begin{equation}
\int_{\Omega}\left|P_{\varphi}\left(f\right)\right|^{2}\left(1-H\right)e^{\psi-\varphi}\leq\int_{\Omega}\left|f\right|^{2}e^{\psi-\varphi},\label{eq:le1}
\end{equation}
for any function $f\in L^{2}\left(\Omega,e^{-\varphi}\right)$.
Here $P_{\varphi}$ denotes the weighted Bergman projection of $L^{2}\left(\Omega,e^{-\varphi}\right)$.

\end{theorem}

\begin{remark}\normalfont
In \cite{BeCh00}, a stronger condition that $i\overline{\partial}\psi\wedge\partial\psi\leq r\,i\partial\overline{\partial}\psi$,
with $r\in\left(0,1\right)$, is imposed. Note also that $\psi$ and $\varphi$ need not be plurisubharmonic in Theorem \ref{BeCh1}.
\end{remark}

The appearance of a general function $H$ in Theorem \ref{BeCh1}
is useful to study the degenerate case in Donnelly-Fefferman theorem
\cite{DoFe83}. It might help to extend or refine several existing
results in this direction. To illustrate some applications of Theorem \ref{BeCh1}, first we obtain the following control on the growth of a negative plurisubharmonic function near the boundary.

\begin{theorem}  \label{cor1}
Let $\Omega\,\subset\mathbb{C}^{n}$ be a bounded pseudoconvex domain.
For each $\varepsilon>0$ and $k\in\mathbb{Z}^{+}$, there does not exist a plurisubharmonic
function $\rho\in PSH^{-}\left(\Omega\right)$ such that
\begin{equation}\label{gr2}
-\rho\left(z\right)\leq\text{\emph{const}}_{\varepsilon,k}\;\delta_{\Omega}\left(z\right)\left(\underbrace{\log\left(\log\ldots\left|\log\delta_{\Omega}\left(z\right)\right|\right)}_{k\;\text{\emph{times}}}\right)^{-\varepsilon},\text{ for any }z\text{ near }\partial\Omega.
\end{equation}
\end{theorem}

\begin{remark}\normalfont 
This is a slight improvement of \cite[Theorem 1.7]{Che14}. 
\end{remark}

Our next result is another interesting application
of Theorem \ref{BeCh1}  in obtaining $L^{2}$  asymptotic behavior of the Bergman kernel $K_{\Omega}$ of $\Omega$.

\begin{theorem}\label{bouBe}

Let $\Omega$ be a bounded pseudoconvex domain in $\mathbb{C}^{n}$, and 
let $\rho\in PSH^{-}\left(\Omega\right)$. Then for any $\alpha>1$ and $k\in\mathbb{N}$,
\[
\lim_{\varepsilon\rightarrow0^{+}}\dfrac{\displaystyle\int_{\left\{ -\rho\leq\varepsilon\right\} }\left|K_{\Omega}\left(\cdot,w\right)\right|^{2}}{\varepsilon\,\left|\log\varepsilon\right|\,\ldots\underbrace{\log\left(\ldots\left|\log\varepsilon\right|\right)}_{k\;\text{\emph{times}}}\,\left(\underbrace{\log\left(\ldots\left|\log\varepsilon\right|\right)}_{k+1\;\text{\emph{times}}}\right)^{\alpha}}\;=\,0,\;\text{ for any } w\in\Omega.
\]

\end{theorem}

\begin{remark}\normalfont
A simple example of the unit ball $\mathbb{B}\subset\mathbb{C}^{n}$
with $\rho\left(z\right)=\left|z\right|-1$ shows that 
\[
\frac{1}{\varepsilon}\int_{\left\{ -\rho\leq\varepsilon\right\} }\left|K_{\mathbb{B}}\left(\cdot,w\right)\right|^{2}\xrightarrow{\varepsilon\rightarrow0^{+}}\text{const}_{n}\int_{\partial \mathbb{B}}\left|K_{\mathbb{B}}\left(\cdot,w\right)\right|^{2}=c\left(w\right)>0,\forall w\in\mathbb{B}.
\]
Note also that, due to an example of Herbort \cite{Her83}, the boundary
behavior of the Bergman kernel may involve  logarithmic terms.

\end{remark}

Our plan of the paper is as follows. In Section \ref{BeCh}, we give the proof of Theorem  \ref{BeCh1}, which will be needed through the paper.  Theorem \ref{cor1} and \ref{bouBe} will be proved in Sect. \ref{coap1} and \ref{coap2}. The estimate of integrability indices in Theorem \ref{thm_int} will be accomplished in Sect. \ref{int_imp}. In the last two sections, we prove the smoothing properties of the Bergman projection and of weighted Bergman projections, Theorem \ref{thm1} and \ref{thm2}.

\section{Proof of Theorem \ref{BeCh1}} \label{BeCh}

We shall employ the idea of twisting $\overline{\partial}$ equations
used in \cite{BeCh00}, see also \cite{Thu18}. For properties of weighted Bergman projections, we refer the reader to \cite{PWZ90}.

\begin{proof}
Assume that the RHS of \eqref{eq:le1} is finite. Assume also for
a moment that $\psi$ and $\varphi$ are in $C^{2}\left(\overline{\Omega}\right)$. Let $u_{\varphi}=h-P_{\varphi}\left(h\right)$, where $h=e^{-\psi}P_{\varphi+\psi}\left(e^{\psi}f\right)$.  We then
have
\[
\overline{\partial}\left(u_{\varphi}e^{\psi}\right)=\left(\overline{\partial}h+u_{\varphi}\overline{\partial}\psi\right)e^{\psi}=\left(u_{\varphi}-e^{-\psi}P_{\varphi+\psi}\left(e^{\psi}f\right)\right)e^{\psi}\overline{\partial}\psi.
\]
Thus
\[
i\overline{\partial}\left(u_{\varphi}e^{\psi}\right)\wedge\partial\left(u_{\varphi}e^{\psi}\right)\leq H\left|u_{\varphi}-e^{-\psi}P_{\varphi+\psi}\left(e^{\psi}f\right)\right|^{2}e^{2\psi}i\partial\overline{\partial}\left(\varphi+\psi\right).
\]
Since $u_{\varphi}e^{\psi}$ is the $L^{2}\left(\Omega,e^{-\varphi-\psi}\right)$-minimal
solution to $\overline{\partial}u=\overline{\partial}\left(u_{\varphi}e^{\psi}\right)$,
by $L^{2}$-H{\"o}rmander estimate, see e.g. \cite[Theorem 2.1]{Blo14},
\begin{equation}
\int_{\Omega}\left|u_{\varphi}\right|^{2}e^{\psi-\varphi}\leq\int_{\Omega}H\left|u_{\varphi}-e^{-\psi}P_{\varphi+\psi}\left(e^{\psi}f\right)\right|^{2}e^{\psi-\varphi}=\int_{\Omega}H\left|P_{\varphi}\left(h\right)\right|^{2}e^{\psi-\varphi}.\label{eq:Ber1}
\end{equation}
On the other hand,
\begin{equation}
\int_{\Omega}\left|P_{\varphi}\left(h\right)\right|^{2}e^{\psi-\varphi}=\int_{\Omega}\left|u_{\varphi}-h\right|^{2}e^{\psi-\varphi}=\int_{\Omega}\left|u_{\varphi}\right|^{2}e^{\psi-\varphi}+\int_{\Omega}\left|h\right|^{2}e^{\psi-\varphi},\label{eq:Ber2}
\end{equation}
where the last equality follows by 
\[
\int_{\Omega}u_{\varphi}\overline{P_{\varphi+\psi}\left(e^{\psi}f\right)}e^{-\varphi}=0.
\]
Since, for any $v\in A^{2}\left(\Omega,e^{-\varphi}\right)=\mathcal{O}\left(\Omega\right)\cap L^{2}\left(\Omega,e^{-\varphi}\right)$, 
\[
\intop_{\Omega}\left(h-f\right)\overline{v}e^{-\varphi}=\int_{\Omega}\left(P_{\varphi+\psi}\left(e^{\psi}f\right)-e^{\psi}f\right)\overline{v}e^{-\left(\varphi+\psi\right)}=0,
\]
we get that $P_{\varphi}\left(h\right)=P_{\varphi}\left(f\right)$.
Note also that 
\begin{equation}
\int_{\Omega}\left|h\right|^{2}e^{\psi-\varphi}=\int_{\Omega}\left|P_{\varphi+\psi}\left(e^{\psi}f\right)\right|^{2}e^{-\left(\varphi+\psi\right)}\leq\int_{\Omega}\left|f\right|^{2}e^{\psi-\varphi}.\label{eq:Ber3}
\end{equation}
The conclusion now follows from \eqref{eq:Ber1}, \eqref{eq:Ber2}, \eqref{eq:Ber3}.

For general functions $\varphi,\psi\in C^{2}\left(\Omega\right)$,
consider a sequence of pseudoconvex domains $\left\{ \Omega_{j}\right\} $
such that $\overline{\Omega_{j}}\Subset\Omega_{j+1}$ and $\Omega=\bigcup_{j=1}^{\infty}\Omega_{j}$.
For each (fixed) $j_{0}$, the estimate above gives
\[
\int_{\Omega_{j_{0}}}\left|P_{\Omega_{j},\varphi}\left(f\right)\right|^{2}\left(1-H\right)e^{\psi-\varphi}\leq\int_{\Omega_{j}}\left|f\right|^{2}e^{\psi-\varphi}\leq\int_{\Omega}\left|f\right|^{2}e^{\psi-\varphi},\;\forall j\geq j_{0}.
\]
Here $P_{\Omega_{j},\varphi}$ is the weighted Bergman projection
of $L^{2}\left(\Omega_{j},e^{-\varphi}\right)$. Using the fact that $\sup_{\Omega_{j_{0}}}H<1$,
the sequence $\left\{ P_{\Omega_{j},\varphi}\left(f\right)\right\} _{j\geq j_{0}}$
is bounded in $L^{2}\left(\Omega_{j_{0}}\right)$. Thus, by Cantor\textquoteright s
diagonal argument, we can assume, by passing to a subsequence, that
$P_{\Omega_{j},\varphi}\left(f\right)$ converges weakly to a function
$v$ in $L_{loc}^{2}\left(\Omega\right)$. In fact, $v=P_{\Omega,\varphi}\left(f\right)$,
to see this, take any $K\Subset\Omega$ and any $h\in A^{2}\left(\Omega,e^{-\varphi}\right)$,
since $e^{\left.-\varphi\right/2}\left(f-P_{\Omega_{j},\varphi}\left(f\right)\right)$
also converges weakly to $e^{\left.-\varphi\right/2}\left(f-v\right)$
in $L^{2}\left(K\right)$, 
\begin{align*}
\int_{K}\left|f-v\right|^{2}e^{-\varphi} & \leq\liminf_{j\rightarrow\infty}\int_{K}\left|f-P_{\Omega_{j},\varphi}\left(f\right)\right|^{2}e^{-\varphi}\\
 & \leq\liminf_{j\rightarrow\infty}\int_{\Omega_{j}}\left|f-P_{\Omega_{j},\varphi}\left(f\right)\right|^{2}e^{-\varphi}\\
 & \leq\liminf_{j\rightarrow\infty}\int_{\Omega_{j}}\left|f-h\right|^{2}e^{-\varphi}\\
 & \leq\int_{\Omega}\left|f-h\right|^{2}e^{-\varphi}.
\end{align*}
It follows that
\[
\int_{\Omega}\left|f-v\right|^{2}e^{-\varphi}\leq\int_{\Omega}\left|f-h\right|^{2}e^{-\varphi},\;\forall h\in A^{2}\left(\Omega,e^{-\varphi}\right),
\]
so $v=P_{\Omega,\varphi}\left(f\right)$. For any $K\Subset\Omega$, since $\left(1-H\right)^{\frac{1}{2}}e^{\frac{\psi-\varphi}{2}}P_{\Omega_{j},\varphi}\left(f\right)$
converges weakly in $L^{2}\left(K\right)$ to $\left(1-H\right)^{\frac{1}{2}}e^{\frac{\psi-\varphi}{2}}P_{\Omega,\varphi}\left(f\right)$,  we conclude
that
\begin{align*}
\int_{K}\left|P_{\Omega,\varphi}\left(f\right)\right|^{2}\left(1-H\right)e^{\psi-\varphi} & \leq\liminf_{j\rightarrow\infty}\int_{K}\left|P_{\Omega_{j},\varphi}\left(f\right)\right|^{2}\left(1-H\right)e^{\psi-\varphi}\\
 & \leq\int_{\Omega}\left|f\right|^{2}e^{\psi-\varphi}.
\end{align*}
Therefore
\[
\int_{\Omega}\left|P_{\Omega,\varphi}\left(f\right)\right|^{2}\left(1-H\right)e^{\psi-\varphi}\leq\int_{\Omega}\left|f\right|^{2}e^{\psi-\varphi}.
\]
\end{proof}

\section{Proof of Theorem \ref{cor1}} \label{coap1}
\begin{proof}
For fixed $\varepsilon$ and $k\in\mathbb{Z}^{+}$, assume for a contradiction
that there exists such a function $\rho$. Then, by modifying $\rho$ in the interior, we may assume that $\rho\in L^{\infty}\left(\Omega\right)$ and  the estimate \eqref{gr2} is true on the whole domain $\Omega$. By scaling, we may also assume that $\text{diam}\left(\Omega\right)$ is small.

Let us assume for a moment that $\rho\in C^{2}\left(\Omega\right)$. We shall apply Theorem  \ref{BeCh1} with $\varphi=0$ and $\psi=-\log\left(-\rho\right)-\Psi\left(-\log\left(-\rho\right)\right):=u-\Psi\left(u\right)$,
where
\[
\Psi\left(t\right)=\frac{\varepsilon}{2}\underbrace{\log\left(\log\ldots\log t\right)}_{k\;\text{times}}:=\frac{\varepsilon}{2}\log_{k}\left(t\right).
\]
For convenience, set $\log_{0}\left(t\right):=t$. Notice that $$i\partial u\wedge\overline{\partial}u\leq i\partial\overline{\partial}u,\quad 
\overline{\partial}\psi=\left(1-\Psi'\left(u\right)\right)\overline{\partial}u,$$
and 
\[
i\partial\overline{\partial}\psi=\left(1-\Psi'\left(u\right)\right)i\partial\overline{\partial}u+\left(-\Psi"\left(u\right)\right)i\partial u\wedge\overline{\partial}u\geq\left(1-\Psi'\left(u\right)\right)i\partial\overline{\partial}u.
\]
Thus 
\[
i\overline{\partial}\psi\wedge\partial\psi\leq\left(1-\Psi'\left(u\right)\right)i\partial\overline{\partial}\psi.
\]
We get that 
\begin{equation}
\int_{\Omega}\left|P_{\Omega}\left(f\right)\right|^{2}\frac{e^{-\Psi\left(u\right)}\Psi'\left(u\right)}{\left(-\rho\right)}\leq\int_{\Omega}\left|f\right|^{2}\frac{e^{-\Psi\left(u\right)}}{\left(-\rho\right)},\label{gr1}
\end{equation}
for any function $f\in L^{2}\left(\Omega\right)$. Now, apply $f:=-\rho$ then the RHS
of \eqref{gr1} is finite. On the other hand,
\[
\frac{e^{-\Psi\left(u\right)}\Psi'\left(u\right)}{\left(-\rho\right)}=\dfrac{\left.\varepsilon\right/2}{\left(-\rho\right)\left(\log_{k-1}\left(\left|\log\left(-\rho\right)\right|\right)\right)^{\left.\varepsilon\right/2}\,\prod_{j=0}^{k-1}\log_{j}\left(\left|\log\left(-\rho\right)\right|\right)},
\]
which is decreasing with respect to $-\rho$. Therefore
\[
\frac{e^{-\Psi\left(u\right)}\Psi'\left(u\right)}{\left(-\rho\right)}\geq\text{const}_{k,\varepsilon}\dfrac{1}{\delta_{\Omega}\prod_{j=0}^{k-1}\log_{j}\left(\left|\log\delta_{\Omega}\right|\right)}.
\]
By the claim below, one gets that $P_{\Omega}\left(-\rho\right)=0$.
It follows that $\int_{\Omega}\rho\overline{1}=0$. This is a contradiction.

For a general $\rho\in PSH^{-}\left(\Omega\right)$, it suffices to
show that the estimate \eqref{gr1} is also true. This can be done
by a limiting process similar to that used in the proof of Theorem \ref{BeCh1}.
We will retain it for clarity. Let $\left\{ \Omega_{j}\right\} $ be a sequence of pseudoconvex domains
such that $\overline{\Omega_{j}}\Subset\Omega_{j+1}$ and $\Omega=\bigcup_{j=1}^{\infty}\Omega_{j}$.
For each $\sigma>0$, let $u_{\sigma}=-\log\left(e^{-u}\star\eta_{\sigma}\right)$,
where $u=-\log\left(-\rho\right)$, and $\star\;\eta_{\sigma}$ denotes the standard
convolution. Then $-e^{-u_{\sigma}}=\rho\star\eta_{\sigma}\in PSH^{-}\left(\Omega\right)$,
thus $i\overline{\partial}u_{\sigma}\wedge\partial u_{\sigma}\leq i\partial\overline{\partial}u_{\sigma}$.
By the monotone convergence theorem, for each $j$, we can choose
$0<\sigma_{j}<\text{dist}\left(\Omega_{j},\partial\Omega\right)$
such that
\[
\int_{\Omega_{j}}\left|f\right|^{2}\left(e^{u_{\sigma_{j}}-\Psi\left(u_{\sigma_{j}}\right)}-e^{u-\Psi\left(u\right)}\right)<\frac{1}{j}.
\]
Since $u_{\sigma_{j}}\in C^{\infty}\left(\overline{\Omega_{j}}\right)$,
\[
\displaystyle\intop_{\Omega_{j}}\left|P_{\Omega_{j}}\left(f\right)\right|^{2}\Psi'\left(u_{\sigma_{j}}\right)e^{u_{\sigma_{j}}-\Psi\left(u_{\sigma_{j}}\right)}\leq\intop_{\Omega_{j}}\left|f\right|^{2}e^{u_{\sigma_{j}}-\Psi\left(u_{\sigma_{j}}\right)}\leq\frac{1}{j}+\intop_{\Omega}\left|f\right|^{2}e^{u-\Psi\left(u\right)}.
\]
Simple calculation shows that $\left(1-\Psi'\left(t\right)\right)\Psi'\left(t\right)+\Psi"\left(t\right)\geq0$
for $t\gg1$. Thus $t\rightarrow\Psi'\left(t\right)e^{t-\Psi\left(t\right)}$
is non-decreasing. It follows that
\[
\Psi'\left(u_{\sigma_{j}}\right)e^{u_{\sigma_{j}}-\Psi\left(u_{\sigma_{j}}\right)}\geq\Psi'\left(u\right)e^{u-\Psi\left(u\right)}\geq\text{const}_{k,\varepsilon}\dfrac{1}{\delta_{\Omega}\prod_{j=0}^{k-1}\log_{j}\left(\left|\log\delta_{\Omega}\right|\right)}.
\]
Therefore, by passing to a subsequence, $P_{\Omega_{j}}\left(f\right)$
converges weakly to $P_{\Omega}\left(f\right)$ in $L_{loc}^{2}\left(\Omega\right)$.
For each $l\in\mathbb{Z}^{+}$, let $v_{l}=\min\left\{ \Psi'\left(u\right)e^{u-\Psi\left(u\right)},l\right\} \in L^{\infty}\left(\Omega\right)$.
Then 
\begin{align*}
\int_{\Omega_{j_{0}}}\left|P_{\Omega}\left(f\right)\right|^{2}v_{l} & \leq\liminf_{j\rightarrow\infty}\int_{\Omega_{j_{0}}}\left|P_{\Omega_{j}}\left(f\right)\right|^{2}v_{l}\\
 & \leq\liminf_{j\rightarrow\infty}\int_{\Omega_{j_{0}}}\left|P_{\Omega_{j}}\left(f\right)\right|^{2}\Psi'\left(u_{\sigma_{j}}\right)e^{u_{\sigma_{j}}-\Psi\left(u_{\sigma_{j}}\right)}\\
 & \leq\int_{\Omega}\left|f\right|^{2}e^{u-\Psi\left(u\right)},\;\forall\, l\in\mathbb{Z}^{+}.
\end{align*}
Now Fatou's lemma gives
\[
\int_{\Omega}\left|P_{\Omega}\left(f\right)\right|^{2}\Psi'\left(u\right)e^{u-\Psi\left(u\right)}\leq\int_{\Omega}\left|f\right|^{2}e^{u-\Psi\left(u\right)}.
\]

\vspace*{0.15cm}

\emph{Claim}: \emph{there is no holomorphic function $f$ on $\Omega$,
except for $f\equiv0$, such that}
\[
\int_{\Omega}\dfrac{\left|f\right|^{2}}{\delta_{\Omega}\prod_{j=0}^{k-1}\log_{j}\left(\left|\log\delta_{\Omega}\right|\right)}<\infty.
\]
\emph{Proof of the claim}. This essentially follows from the proof
of \cite[Theorem 1.6]{Che14}. The key is to apply the Bochner-Martinelli formula:

\textbf{Bochner-Martinelli formula}: Let $D\subset\mathbb{C}^{n}$ be a bounded
domain with $C^{1}$ boundary, and let $f\in C^{1}\left(\overline{D}\right)$.
Then for each $z\in D$,
\begin{align*}
f\left(z\right) & =\frac{\left(n-1\right)!}{\left(2\pi i\right)^{n}}\biggl(\int_{\partial D}f\left(\varsigma\right)\sum_{j=1}^{n}\frac{\left(-1\right)^{j+1}\left(\overline{\varsigma}_{j}-\overline{z}_{j}\right)}{\left|\varsigma-z\right|^{2n}}\widehat{d\overline{\varsigma}_{j}}\wedge d\varsigma\\
& \hspace*{4.5cm} -\int_{D}\sum_{j=1}^{n}\left(\overline{\varsigma}_{j}-\overline{z}_{j}\right)\frac{\partial f}{\partial\overline{\varsigma}_{j}}\left(\varsigma\right)\frac{d\overline{\varsigma}\wedge d\varsigma}{\left|\varsigma-z\right|^{2n}}  \biggr).
\end{align*}
Here $\widehat{d\overline{\varsigma}_{j}}=d\overline{\varsigma}_{1}\wedge\ldots\wedge d\overline{\varsigma}_{j-1}\wedge d\overline{\varsigma}_{j+1}\wedge\ldots\wedge d\overline{\varsigma}_{n}$.

Since $$\left|\delta_{\Omega}\left(z\right)-\delta_{\Omega}\left(w\right)\right|\leq\left|z-w\right|,\forall z,w\in \mathbb{C}^{n},$$
one has $\left\Vert \bigtriangledown\delta_{\varepsilon}\right\Vert _{L^{\infty}\left(\Omega\right)}\leq\text{const}_{n}$.
Here $\delta_{\varepsilon}:=\delta_{\Omega}\star\eta_{\varepsilon}$
is the standard convolution. Fix any $z_{0}\in\Omega$. For each $\varepsilon>0$,
choose $\varepsilon_{1}>0$ such that $\left\Vert \delta_{\Omega}-\delta_{\varepsilon_{1}}\right\Vert _{L^{\infty}\left(\Omega\right)}\leq\left.\varepsilon\right/10.$
It is clear that
\begin{align*}
\left\{ z\in\Omega:\log\frac{1}{2}<\log_{k}\left(\left|\log\delta_{\varepsilon_{1}}\left(z\right)\right|\right)-\log_{k}\left(\left|\log\varepsilon\right|\right)<0\right\} \\
\subset\left\{ z\in\Omega:\frac{9}{10}\varepsilon<\delta_{\Omega}\left(z\right)<c_{\varepsilon}\right\} ,
\end{align*}
for some constant $c_{\varepsilon}$ that converges to $0$ as $\varepsilon\rightarrow0$.
And also $\frac{8}{9}\delta_{\Omega}\left(z\right)<\delta_{\varepsilon_{1}}\left(z\right)<\frac{10}{9}\delta_{\Omega}\left(z\right)$,
provided that $\frac{9}{10}\varepsilon<\delta_{\Omega}\left(z\right)$.
Choose a cut-off function $\chi$ on $\mathbb{R}$ such that $\chi\left(t\right)=1$
for $t<\log\frac{1}{2}$ and $\chi\left(t\right)=0$ for $t>0$. For each $\varepsilon>0$ (small), consider a smooth domain $D_{\varepsilon}\Subset\Omega$
so that $\delta_{\Omega}\left(z\right)<\frac{1}{10}\varepsilon$,
for any $z\in\partial D_{\varepsilon}$. Then apply the Bochner-Martinelli
formula to the domain $D_{\varepsilon}$ and the function
\[
g_{\varepsilon}\left(z\right)=\chi\left(\log_{k}\left(\left|\log\delta_{\varepsilon_{1}}\left(z\right)\right|\right)-\log_{k}\left(\left|\log\varepsilon\right|\right)\right)f^{2}\left(z\right).
\]
Since $g_{\varepsilon}=0$ on $\partial D_{\varepsilon}$, and $f$
is holomorphic on $\Omega$, we get that
\begin{align*}
\left|f\left(z_{0}\right)\right|^{2} &\leq\text{c}_{n}\intop_{\Omega}\hspace*{-0.1cm}\dfrac{\left|f\right|^{2}\left(\varsigma\right)\left|\chi'\left(\cdot\right)\right|}{\delta_{\varepsilon_{1}}\left(\varsigma\right)\displaystyle\prod_{j=0}^{k-1}\log_{j}\left(\left|\log\delta_{\varepsilon_{1}}\left(\varsigma\right)\right|\right)}\left|\sum_{j=1}^{n}\left(\overline{\varsigma}_{j}-\overline{z_{0}}_{j}\right)\frac{\partial\delta_{\varepsilon_{1}}}{\partial\overline{\varsigma}_{j}}\left(\varsigma\right)\right|\frac{d\overline{\varsigma}\wedge d\varsigma}{\left|\varsigma-z_{0}\right|^{2n}}\\
 & \leq\text{c}_{n,z_{0}}\;\int_{\left\{ \frac{9}{10}\varepsilon<\delta_{\Omega}\left(\cdot\right)<c_{\varepsilon}\right\} }\dfrac{\left|f\right|^{2}\left(\varsigma\right)}{\delta_{\Omega}\left(\varsigma\right)\prod_{j=0}^{k-1}\log_{j}\left(\left|\log\delta_{\Omega}\left(\varsigma\right)\right|\right)}d\overline{\varsigma}\wedge d\varsigma\\
 & \rightarrow0\text{ as }\varepsilon\rightarrow0,
\end{align*}
given that

\[
\int_{\Omega}\dfrac{\left|f\right|^{2}}{\delta_{\Omega}\prod_{j=0}^{k-1}\log_{j}\left(\left|\log\delta_{\Omega}\right|\right)}<\infty.
\]
This ends the proof of the claim, and so the proof of Theorem \ref{cor1} is complete.

\end{proof}

\section{Proof of Theorem \ref{bouBe}} \label{coap2}

\begin{proof}
Again, we may assume that $\rho\in C^{2}\left(\Omega\right)$. By modifying $\rho$ on the region where $-\rho$ is large, we may also 
assume without loss of generality that $\left\Vert \rho\right\Vert _{L^{\infty}\left(\Omega\right)}$
is small. Let $$\varphi=0 \;\text{ and }\;\psi=-\log\left(-\rho+\varepsilon\right)-\Psi\left(-\log\left(-\rho+\varepsilon\right)\right),$$
where
\[
\Psi\left(t\right)=\frac{\alpha-1}{2}\underbrace{\log\left(\log\ldots\log t\right)}_{k\;\text{times}}:=\frac{\alpha-1}{2}\log_{k}\left(t\right).
\]
Theorem \ref{BeCh1} gives 
\[
\intop_{\Omega}\frac{\left(\alpha-1\right)\left|P_{\Omega}\left(f\right)\right|^{2}}{2\left(-\rho+\varepsilon\right)\left(\displaystyle\prod_{j=0}^{k-2}\log_{j}\left(u\right)\right)\left(\log_{k-1}\left(u\right)\right)^{\frac{\alpha+1}{2}}}\leq\intop_{\Omega}\dfrac{\left|f\right|^{2}}{\left(-\rho+\varepsilon\right)\left(\log_{k-1}\left(u\right)\right)^{\frac{\alpha-1}{2}}},
\]
where $u:=-\log\left(-\rho+\varepsilon\right)$. For a fixed $w\in\Omega$,
apply $f\left(\cdot\right)=\chi_{W}K_{W}\left(\cdot,w\right)$; where
$\chi_{W}$ is the indicator function of $W:=\left\{ -\rho>\frac{-\rho\left(w\right)}{2}\right\} $,
and $K_{W}$ is the Bergman kernel of $W$, we get that
\[
\intop_{\Omega}\frac{\left(\alpha-1\right)\left|K_{\Omega}\left(\cdot,w\right)\right|^{2}}{2\left(-\rho+\varepsilon\right)\left(\displaystyle\prod_{j=0}^{k-2}\log_{j}\left(u\right)\right)\left(\log_{k-1}\left(u\right)\right)^{\frac{\alpha+1}{2}}}\leq\intop_{W}\dfrac{\left|K_{W}\left(\cdot,w\right)\right|^{2}}{\left(-\rho+\varepsilon\right)\left(\log_{k-1}\left(u\right)\right)^{\frac{\alpha-1}{2}}}.
\]
Since $\alpha>1$ and 
\[
\int_{W}\dfrac{\left|K_{W}\left(\cdot,w\right)\right|^{2}}{\left(-\rho+\varepsilon\right)\left(\log_{k-1}\left(u\right)\right)^{\frac{\alpha-1}{2}}}<c\left(w\right),\forall\varepsilon>0,
\]
the conclusion follows.

\end{proof}

\section{Proof of Theorem \ref{thm_int}} \label{int_imp}

We first recall some properties of the pluricomplex Green function. 

Let $\Omega$ be a bounded domain in $\mathbb{C}^{n}$. The pluricomplex
Green function $g_{\Omega}\left(\cdot,w\right)$ with pole at $w\in\Omega$
is defined by 
\[
g_{\Omega}\left(z,w\right):=\sup\left\{ u\left(z\right):u\in PSH^{-}\left(\Omega\right),\limsup_{\xi\rightarrow w}\left(u\left(\xi\right)-\log\left|\xi-w\right|\right)<\infty\right\} .
\]
It is known that (see \cite{Blo05}) when $\Omega$ is pseudoconvex
and $f$ is holomorphic on $\Omega$, 
\begin{equation}
\left|f\left(w\right)\right|^{2}\leq\text{const}_{n}\,K_{\Omega}\left(w,w\right)\int_{\left\{ g_{\Omega}\left(\cdot,w\right)<-1\right\} }\left|f\right|^{2},\;\forall w\in\Omega.\label{eq:mn1}
\end{equation}
We also need the following well-known estimate of $g_{\Omega}$:

\begin{theorem}
\label{thm:mn2}\emph{(B{\l}ocki \cite{Blo05})} Let $\Omega$ be a bounded
domain in $\mathbb{C}^{n}$. Assume that there is $v\in PSH^{-}\left(\Omega\right)$
such that 
\[
C_{1}\delta_{\Omega}^{t}\left(z\right)\leq-v\left(z\right)\leq C_{2}\delta_{\Omega}^{t}\left(z\right),\quad z\in\Omega,
\]
for some positive constants $C_{1},C_{2}$ and $t$. Then there exist
positive constants $C$ and $\delta_{0}$ such that
\[
\left\{ g_{\Omega}\left(\cdot,w\right)<-1\right\}  \hspace*{-0.05cm} \subset  \hspace*{-0.05cm} \left\{ \hspace*{-0.05cm} \frac{1}{C}\delta_{\Omega}\left(w\right)\left|\log\delta_{\Omega}\left(w\right)\right|^{-\frac{1}{t}}\leq\delta_{\Omega}\left(\cdot\right)\leq C\delta_{\Omega}\left(w\right)\left|\log\delta_{\Omega}\left(w\right)\right|^{\frac{n}{t}} \hspace*{-0.05cm}\right\}  \hspace*{-0.05cm},
\]
for any $w\in\Omega$ with $\delta_{\Omega}\left(w\right)<\delta_{0}$.
\end{theorem}

We now give the proof of Theorem \ref{thm_int}. The key to gaining the exponent in \eqref{improChen}, compared to the estimate \eqref{Chen1}, is the fact that we can exploit the smoothness of $\partial{\Omega}$ on the Bergman kernel by using the estimate \eqref{eq:mn1}.

\begin{proof}[Proof of Theorem \ref{thm_int}]
For each $t<\alpha\left(\Omega\right)$, by definition, there is $\rho\in PSH^{-}\left(\Omega\right)$
such that $-\rho\leq C\delta_{\Omega}^{t}$ on $\Omega$. For any
(fixed) $w\in\Omega$ and every $0<r<1$, by Theorem \ref{thm:mn2}
and the estimate \eqref{eq:mn1}, we conclude that for any $z\in\Omega$
with $\delta_{\Omega}\left(z\right)<\delta_{0}$,
\begin{align*}
\left|K_{\Omega}\left(z,w\right)\right|^{2} & \leq\text{c}_{n}\;K_{\Omega}\left(z,z\right)\int_{\left\{ g_{\Omega}\left(\cdot,z\right)<-1\right\} }\left|K_{\Omega}\left(\cdot,w\right)\right|^{2}\\
 & \leq\text{c}_{\Omega}\;\delta_{\Omega}^{-n-1}\left(z\right)\left(\delta_{\Omega}\left(z\right)\left|\log\delta_{\Omega}\left(z\right)\right|^{\frac{n}{c}}\right)^{tr}\hspace*{-0.4cm} \displaystyle\intop_{\left\{ g_{\Omega}\left(\cdot,z\right)<-1\right\} }\hspace*{-0.3cm} \left|K_{\Omega}\left(\cdot,w\right)\right|^{2}\left(-\rho\right)^{-r}\\
 & \leq\text{c}_{\Omega}\;\delta_{\Omega}^{-n-1+tr}\left(z\right)\left|\log\delta_{\Omega}\left(z\right)\right|^{\frac{ntr}{c}}\left(-\rho\left(w\right)\right)^{-r}K_{\left\{ -\rho>\frac{-\rho\left(w\right)}{2}\right\} }\left(w,w\right).
\end{align*}
Here we have used the fact that there are $v\in PSH^{-}\left(\Omega\right)$
and $c>0$ such that $C_{1}\delta_{\Omega}^{c}\leq-v\leq C_{2}\delta_{\Omega}^{c}$
on $\Omega$. The last inequality above follows by plugging $f\left(\cdot\right)=\chi_{\left\{ -\rho>\frac{-\rho\left(w\right)}{2}\right\} }\left(\cdot\right)K_{\left\{ -\rho>\frac{-\rho\left(w\right)}{2}\right\} }\left(\cdot,w\right)$
into the estimate (in which we apply Theorem \ref{BeCh1} to $\psi=-r\log\left(-\rho\right)$
and $\varphi=0$; remark also that this is true for a general plurisubharmonic function $\rho$)
\[
\int_{\Omega}\left|P_{\Omega}\left(f\right)\right|^{2}\left(-\rho\right)^{-r}\leq\frac{1}{1-r}\int_{\Omega}\left|f\right|^{2}\left(-\rho\right)^{-r}.
\]
Thus 
\[
\left|K_{\Omega}\left(z,w\right)\right|^{2}\leq\text{const}_{w}\;\delta_{\Omega}^{-n-1+tr}\left(z\right)\left|\log\delta_{\Omega}\left(z\right)\right|^{\frac{ntr}{c}},\;\forall z\in\Omega,\text{ with }\delta_{\Omega}\left(z\right)<\delta_{0}.
\]
On the other hand, for any $z\in\Omega$ with $\delta_{\Omega}\left(z\right)\geq\delta_{0}$,
\begin{align*}
\left|K_{\Omega}\left(z,w\right)\right|^{2} & \leq\text{c}_{n}\;K_{\Omega}\left(z,z\right)\int_{\left\{ g_{\Omega}\left(\cdot,z\right)<-1\right\} }\left|K_{\Omega}\left(\cdot,w\right)\right|^{2}\\
 & \leq\text{c}_{\Omega}\;\delta_{0}^{-n-1}K_{\Omega}\left(w,w\right):=\text{const}_{w}.
\end{align*}
Similarly, for each $\varepsilon>0$,
\begin{align*}
\int_{\left\{ \delta_{\Omega}\leq\varepsilon\right\} }\left|K_{\Omega}\left(\cdot,w\right)\right|^{2} & \leq\int_{\left\{ -\rho\leq C\varepsilon^{t}\right\} }\left|K_{\Omega}\left(\cdot,w\right)\right|^{2}\\
 & \leq\text{const}\;\varepsilon^{tr}\int_{\left\{ -\rho\leq C\varepsilon^{t}\right\} }\left|K_{\Omega}\left(\cdot,w\right)\right|^{2}\left(-\rho\right)^{-r}\\
 & \leq\text{const}_{w}\;\varepsilon^{tr}.
\end{align*}
Therefore, for any $\beta\geq2$, we get that
\begin{align*}{\displaystyle \;\intop_{\Omega}\left|K_{\Omega}\left(\cdot,w\right)\right|^{\beta}}\\
 & \hspace*{-1.6cm} \leq{\displaystyle\; \intop_{\left\{ \delta_{\Omega}\geq\delta_{0}\right\} }\left|K_{\Omega}\left(\cdot,w\right)\right|^{\beta}+\sum_{j=0}^{\infty}\;{\displaystyle \intop_{\left\{ 2^{-j-1}\delta_{0}\leq\delta_{\Omega}<2^{-j}\delta_{0}\right\} }\left|K_{\Omega}\left(\cdot,w\right)\right|^{\beta}}}\\
 &  \hspace*{-1.6cm} \leq\text{c}_{w}\left(1+\sum_{j=0}^{\infty}2^{\left(j+1\right)\left(n+1-tr\right)\left(\frac{\beta-2}{2}\right)}\left|\log\left(2^{-j-1}\delta_{0}\right)\right|^{\frac{ntr\left(\beta-2\right)}{2c}}\hspace*{-0.4cm}{\displaystyle \intop_{\left\{ \delta_{\Omega}<2^{-j}\delta_{0}\right\} }\hspace*{-0.5cm}\left|K_{\Omega}\left(\cdot,w\right)\right|^{2}}\right)\\
 &  \hspace*{-1.6cm} \leq\text{c}_{w}\left(1+\sum_{j=0}^{\infty}2^{j\left(\left(n+1-tr\right)\left(\frac{\beta-2}{2}\right)-tr\right)}\left|\log\left(2^{-j-1}\delta_{0}\right)\right|^{\frac{ntr\left(\beta-2\right)}{2c}}\right).
\end{align*}
It follows that if $\left(n+1-tr\right)\left(\frac{\beta-2}{2}\right)-tr<0$
then $K_{\Omega}\left(\cdot,w\right)\in L^{\beta}\left(\Omega\right)$.
This condition is equivalent to 
\[
\beta<2+\dfrac{2tr}{n+1-tr}.
\]
Since this is true for any $t<\alpha\left(\Omega\right)$ and $0<r<1$,
the desired estimate follows.

\end{proof}

A similar approach can be used to obtain the following low-level $L^{p}$
regularity of the Bergman projection, which improves the exponent in \cite[Corollary 2.5]{Thu18}, see also
\cite[Lemma 2.2]{ChZe17}.

\begin{proposition} \label{cor3}
Let $\Omega$ be a bounded pseudoconvex domain with $C^{2}$ boundary
in $\mathbb{C}^{n}$. For any
\[
2\leq q<2+\frac{2\alpha\left(\Omega\right)}{n+1-\alpha\left(\Omega\right)},
\]
and
\[
p>\frac{2}{1-\left(n+1\right)\left(1-\frac{2}{q}\right)},
\]
the Bergman projection $P_{\Omega}$ is bounded from $L^{p}\left(\Omega\right)$
to $L^{q}\left(\Omega\right)$.

\end{proposition}

It is interesting to know how large the $L^{q}$-range of the output
space can be, provided the $L^{p}$-regularity of the input is good enough.
The result in \cite{EdMcN17} indicates that for pseudoconvex domains
with irregular boundary, $q=2$ might be optimal, even if $p=\infty$.

\begin{proof}
The argument is almost identical to that in Theorem \ref{thm_int}, that is, instead of using the estimate 
\[
\left|K_{\Omega}\left(z,w\right)\right|^{2}\leq\text{const}_{n}\;K_{\Omega}\left(z,z\right)\int_{\left\{ g_{\Omega}\left(\cdot,z\right)<-1\right\} }\left|K_{\Omega}\left(\cdot,w\right)\right|^{2},
\]
we consider the estimate
\[
\left|P_{\Omega}\left(f\right)\left(z\right)\right|^{2}\leq\text{const}_{n}\;K_{\Omega}\left(z,z\right)\int_{\left\{ g_{\Omega}\left(\cdot,z\right)<-1\right\} }\left|P_{\Omega}\left(f\right)\right|^{2}.
\]
We leave the details to the reader, see also the proof of \cite[Corollary 2.5]{Thu18}.

\end{proof}

\section{Proof of Theorem \ref{thm1}}

We first prove the following modified version of \cite[Theorem 1.2]{Herb13},  which is perhaps of independent interest in future applications.

\begin{theorem} \label{lem2}
Let $\Omega$ be a bounded domain in $\mathbb{C}^{n}$ with smooth
boundary. Let $\delta\in C^{\infty}\left(\overline{\Omega}\right)$
be a function that equals $\delta_{\Omega}$ near $\partial\Omega$.
For any $\alpha > -1$, any $k\in\mathbb{Z}^{+}$ and every $g\in C^{\infty}\left(\overline{\Omega}\right)$,
there exists a function $\omega\in L^{\infty}\left(\Omega\right)$
such that
\[
\int_{\Omega}hg\delta^{\alpha}=\int_{\Omega}h\omega\delta^{k},
\]
for any holomorphic function $h$ on $\Omega$ satisfying $h\delta^{\alpha}\in L^{1}\left(\Omega\right)$.

\end{theorem}

\begin{remark}\normalfont
We always assume that $\delta$ is positive on $\Omega$.
\end{remark}

\begin{proof}
We employ the method used in \cite[Theorem 1.2]{Herb13}. The idea
is as follows: by acting on the normal direction and using integration
by parts, one can gain some additional powers of $\delta$, but has
to transfer some derivatives in the normal direction on $h$. This
however can be resolved via the holomorphic property of $h$.

Choose a neighborhood $U$ of $\partial\Omega$ such that: $\delta=\delta_{\Omega}$
on $U\cap\Omega$, and the vector field 
\[
N:=\sum_{j=1}^{2n}\frac{\partial\delta_{\Omega}}{\partial x_{j}}\frac{\partial}{\partial x_{j}}
\]
is a smooth vector field on $U\cap\Omega$ satisfying $N\left(\delta_{\Omega}\right)=1$
on $U\cap\Omega$. Here $\left(z_{1},\ldots,z_{n}\right)=\left(x_{1},x_{2},\ldots,x_{2n-1},x_{2n}\right)$
are the standard coordinates of $\mathbb{C}^{n}$. Such a neighborhood exists since $\partial\Omega\in C^{\infty}$. Choose a cut-off
function $\Theta\in C_{0}^{\infty}\left(U\right)$ such that $\Theta=1$
on some neighborhood $U'\Subset U$ of $\partial\Omega$. 

First, we
claim that for given $\alpha > -1$ and $g\in C^{\infty}\left(\overline{\Omega}\right)$,
there is $g_{1}\in C^{\infty}\left(\overline{\Omega}\right)$ such
that 
\begin{equation}
\int_{\Omega}hg\delta^{\alpha}=\int_{\Omega}hg_{1}\delta^{\alpha+1},\label{eq:gh1}
\end{equation}
for any holomorphic function $h$ on $\Omega$ satisfying $h\delta^{\alpha}\in L^{1}\left(\Omega\right)$.
For $\varepsilon>0$ small, let $\Omega_{\varepsilon}:=\left\{ z\in\Omega:\delta_{\Omega}\left(z\right)>\varepsilon\right\} $.
Then 
\begin{align*}
\int_{\Omega_{\varepsilon}}\delta^{\alpha}hg & =\int_{\Omega_{\varepsilon}}\Theta\delta^{\alpha}hg+\int_{\Omega_{\varepsilon}}\left(1-\Theta\right)\delta^{\alpha}hg\\
 & =\frac{1}{\alpha+1}\underbrace{\int_{\Omega_{\varepsilon}}N\left(\delta^{\alpha+1}\right)\Theta hg}_{I_{1}}+\underbrace{\int_{\Omega_{\varepsilon}}h\left(\frac{1-\Theta}{\delta}g\right)\delta^{\alpha+1}}_{Y_{1}}.
\end{align*}
Since the integrand of $Y_{1}$ is in $L^{1}\left(\Omega\right)$,
\[
Y_{1}\xrightarrow{\varepsilon\rightarrow0}\int_{\Omega}h\left(\frac{1-\Theta}{\delta}g\right)\delta^{\alpha+1},
\]
which is in the form of the RHS of \eqref{eq:gh1}. For $I_{1}$:
\begin{align*}
I_{1} & =\int_{\Omega_{\varepsilon}}N\left(\delta^{\alpha+1}\right)\Theta hg=\int_{\Omega_{\varepsilon}}N\left(\delta^{\alpha+1}\Theta hg\right)-\int_{\Omega_{\varepsilon}}\delta^{\alpha+1}N\left(\Theta hg\right)\\
 & =\underbrace{\int_{\Omega_{\varepsilon}}\sum_{j=1}^{2n}\frac{\partial}{\partial x_{j}}\left(\delta_{x_{j}}\delta^{\alpha+1}\Theta hg\right)}_{I_{2}}-\underbrace{\int_{\Omega_{\varepsilon}}\left(\delta^{\alpha+1}\Theta hg\right)\triangle\delta}_{Y_{2}}-\underbrace{\int_{\Omega_{\varepsilon}}\delta^{\alpha+1}N\left(\Theta hg\right)}_{I_{3}}.
\end{align*}
The terms $I_{2}$ and $I_{3}$ can be handled quite similarly. The
key is to transform the action $N\left(h\right)$ into $iT\left(h\right)$,
where $T$ is the tangential vector field defined by
\[
T=\sum_{j=1}^{n}\left(\frac{\partial\delta}{\partial x_{2j}}\,\frac{\partial}{\partial x_{2j-1}}-\frac{\partial\delta}{\partial x_{2j-1}}\,\frac{\partial}{\partial x_{2j}}\right).
\]
One can check that $N\left(h\right)=iT\left(h\right)$, for any $h\in\mathcal{O}\left(\Omega\right)$.
Also, by Stokes' theorem, $\int_{\Omega_{\varepsilon}}T\left(g\right)=0,$
for any $g\in C^{1}\left(\overline{\Omega}_{\varepsilon}\right)$. Thus
\begin{align*}
I_{2} & =\intop_{\Omega_{\varepsilon}}d\left(\delta^{\alpha+1}\Theta hg\sum_{j=1}^{2n}\left(-1\right)^{j+1}\delta_{x_{j}}d\hat{x}_{j}\right)=\varepsilon^{\alpha+1}\intop_{\partial\Omega_{\varepsilon}}\Theta hg\sum_{j=1}^{2n}\left(-1\right)^{j+1}\delta_{x_{j}}d\hat{x}_{j}\\
 & =\varepsilon^{\alpha+1}\intop_{\Omega_{\varepsilon}}d\left(\Theta hg\sum_{j=1}^{2n}\left(-1\right)^{j+1}\delta_{x_{j}}d\hat{x}_{j}\right)\\
 & =\varepsilon^{\alpha+1}\left(\;\intop_{\Omega_{\varepsilon}}N\left(\Theta g\right)h+\intop_{\Omega_{\varepsilon}}N\left(h\right)\Theta g+\intop_{\Omega_{\varepsilon}}\Theta hg\triangle\delta\right).
\end{align*}
One has
\[
\varepsilon^{\alpha+1}\left|\int_{\Omega_{\varepsilon}}N\left(\Theta g\right)h\right|\leq\varepsilon^{\min\left\{ 1,1+\alpha\right\}}\int_{\Omega_{\varepsilon}}\left|N\left(\Theta g\right)\right|\left|h\delta^{\alpha}\right|\xrightarrow{\varepsilon\rightarrow0}0,
\]
since $\delta^{-\alpha}\leq\text{const}_{\Omega}\,\left(1+\varepsilon^{-\alpha}\right)$
on $\Omega_{\varepsilon}$,
 $\Theta g\in C^{\infty}\left(\overline{\Omega}\right)$ and
$h\delta^{\alpha}\in L^{1}\left(\Omega\right)$. Similarly, 
\[
\varepsilon^{\alpha+1}\int_{\Omega_{\varepsilon}}\Theta hg\triangle\delta\xrightarrow{\varepsilon\rightarrow0}0.
\]
On the other hand 
\begin{align*}
\varepsilon^{\alpha+1}\intop_{\Omega_{\varepsilon}}N\left(h\right)\Theta g & =\varepsilon^{\alpha+1}\intop_{\Omega_{\varepsilon}}iT\left(h\right)\Theta g=i\varepsilon^{\alpha+1}\left(\;\intop_{\Omega_{\varepsilon}}T\left(h\Theta g\right)-\intop_{\Omega_{\varepsilon}}hT\left(\Theta g\right)\right)\\
 & =-i\varepsilon^{\alpha+1}\left(\intop_{\Omega_{\varepsilon}}hT\left(\Theta g\right)\right)\rightarrow0\text{ as }\varepsilon\rightarrow0.
\end{align*}
Thus $I_{2}$ converges to $0$ as $\varepsilon\rightarrow0$. For
$I_{3}$, 
\begin{align*}
I_{3} & =\int_{\Omega_{\varepsilon}}\delta^{\alpha+1}N\left(h\right)\Theta g+\int_{\Omega_{\varepsilon}}\delta^{\alpha+1}N\left(\Theta g\right)h\\
 & =\int_{\Omega_{\varepsilon}}i\delta^{\alpha+1}T\left(h\right)\Theta g+\int_{\Omega_{\varepsilon}}\delta^{\alpha+1}N\left(\Theta g\right)h.
\end{align*}
Since $T\left(\delta\right)=0$, we get that $\delta^{\alpha+1}T\left(h\right)\Theta g=T\left(\delta^{\alpha+1}\Theta gh\right)-\delta^{\alpha+1}hT\left(\Theta g\right)$.
Therefore
\begin{align*}
I_{3} & =\int_{\Omega_{\varepsilon}}-i\delta^{\alpha+1}hT\left(\Theta g\right)+\int_{\Omega_{\varepsilon}}\delta^{\alpha+1}N\left(\Theta g\right)h\\
 & =\underbrace{\int_{\Omega_{\varepsilon}}\delta^{\alpha+1}\left(-iT\left(\Theta g\right)+N\left(\Theta g\right)\right)h}_{Y_{3}}.
\end{align*}
The desired claim now follows since
\[
\int_{\Omega_{\varepsilon}}\delta^{\alpha}hg=\frac{1}{\alpha+1}\left(I_{2}-Y_{2}-Y_{3}\right)+Y_{1},
\]
and $Y_{1}$, $Y_{2}$, $Y_{3}$ converge to integrals in the form
of the RHS of \eqref{eq:gh1}. To complete the proof of this lemma,
we simply repeat the above argument, that is
\[
\int_{\Omega}hg\delta^{\alpha}=\int_{\Omega}hg_{1}\delta^{\alpha+1},
\]
and 
\begin{align*}
\int_{\Omega_{\varepsilon}}hg_{1}\delta^{\alpha+1} & =\int_{\Omega_{\varepsilon}}\Theta\delta^{\alpha+1}hg_{1}+\int_{\Omega_{\varepsilon}}\frac{1-\Theta}{\delta}\delta^{\alpha+2}hg_{1}\\
 & =\frac{1}{\alpha+2}\int_{\Omega_{\varepsilon}}N\left(\delta^{\alpha+2}\right)\Theta hg_{1}+\int_{\Omega_{\varepsilon}}\frac{1-\Theta}{\delta}\delta^{\alpha+2}hg_{1}.
\end{align*}
Note that $h\delta^{\alpha+1}\in L^{1}\left(\Omega\right)$. Then
one can find $g_{2}\in C^{\infty}\left(\overline{\Omega}\right)$
so that
\[
\int_{\Omega}hg_{1}\delta^{\alpha+1}=\int_{\Omega}hg_{2}\delta^{\alpha+2}.
\]
So the process can be continued.

\end{proof}

\begin{proof}[Proof of Theorem \ref{thm1}]
Take any $\overline{f}\in\mathcal{O}\left(\Omega\right)\cap L^{2}\left(\Omega\right)$.
Apply $\psi=-\frac{t}{\eta}\log\left(-\rho\right)$ and $\varphi=0$
in Theorem \ref{BeCh1}, we get that
\begin{equation}
\int_{\Omega}\left|P_{\Omega}\left(v\right)\right|^{2}\left(-\rho\right)^{-\frac{t}{\eta}}\leq\frac{\eta}{\eta-t}\int_{\Omega}\left|v\right|^{2}\left(-\rho\right)^{-\frac{t}{\eta}},\label{eq:mn2}
\end{equation}
for any $v\in L^{2}\left(\Omega\right)$. By duality, we also have
\begin{equation}
\int_{\Omega}\left|P_{\Omega}\left(v\right)\right|^{2}\left(-\rho\right)^{\frac{t}{\eta}}\leq\frac{\eta}{\eta-t}\int_{\Omega}\left|v\right|^{2}\left(-\rho\right)^{\frac{t}{\eta}}.\label{eq:mn3}
\end{equation}
Take any $\phi\in C_{0}^{\infty}\left(\Omega\right)$ so that $\left\Vert \phi\right\Vert _{L^{2}\left(\Omega\right)}=1$.
Let $\delta\in C^{\infty}\left(\overline{\Omega}\right)$ be a function
that equals $\delta_{\Omega}$ near $\partial\Omega$. Since $\delta^{-\left.t\right/2}\phi\in L^{2}\left(\Omega\right)$,
by the self-adjointness of $P_{\Omega}$
\begin{equation}\label{combi}
\left|\int_{\Omega}P_{\Omega}\left(fg\right)\delta^{-\frac{t}{2}}\overline{\phi}\right|=\left|\int_{\Omega}fg\overline{P_{\Omega}\left(\delta^{-\frac{t}{2}}\phi\right)}\right|=\left|\int_{\Omega}\overline{f}P_{\Omega}\left(\delta^{-\frac{t}{2}}\phi\right)\overline{g}\right|.
\end{equation}
Since $\overline{f}P_{\Omega}\left(\delta^{-\left.t\right/2}\phi\right)$
is holomorphic and in $L^{1}\left(\Omega\right)$, by Theorem \ref{lem2}, for
any (fixed) $m\in\mathbb{Z}^{+}$, there is $\omega_{m,g}\in L^{\infty}\left(\Omega\right)$
such that
\begin{equation}
\int_{\Omega}\overline{f}P_{\Omega}\left(\delta^{-\frac{t}{2}}\phi\right)\overline{g}=\int_{\Omega}\overline{f}P_{\Omega}\left(\delta^{-\frac{t}{2}}\phi\right)\omega_{m,g}\delta^{m}.\label{eq:mn7}
\end{equation}
By \eqref{eq:mn1},
\[
\left|P_{\Omega}\left(\delta^{-\frac{t}{2}}\phi\right)\left(z\right)\right|^{2}\leq\text{const}_{n}\,K_{\Omega}\left(z,z\right)\int_{\left\{ g_{\Omega}\left(\cdot,z\right)<-1\right\} }\left|P_{\Omega}\left(\delta^{-\frac{t}{2}}\phi\right)\right|^{2}.
\]
One has  $K_{\Omega}\left(z,z\right)\leq\text{const}_{\Omega}\,\delta^{-n-1}\left(z\right)$,
$z\in\Omega$. By Hopf lemma (see \cite[Proposition 12.2]{FoSte87}), $-\rho\geq\text{const}_{\Omega}\,\delta$
on $\Omega$. Also, from Theorem \ref{thm:mn2}, we get that $\delta\geq\text{const}_{\Omega}\;\delta\left(z\right)\left|\log\delta\left(z\right)\right|^{-\frac{1}{c}}$ on $\left\{ g_{\Omega}\left(\cdot,z\right)<-1\right\} $, for some positive constant $c$ and for any
$z\in\Omega$ with $\delta\left(z\right)<\delta_{0}$.
Therefore
\[
\left|P_{\Omega}\left(\delta^{-\frac{t}{2}}\phi\right)\left(z\right)\right|^{2}\leq\text{const}_{\Omega}\,\delta^{-n-1-\frac{t}{\eta}}\left(z\right)\left|\log\delta\left(z\right)\right|^{\frac{t}{c\eta}}\int_{\Omega}\left|P_{\Omega}\left(\delta^{-\frac{t}{2}}\phi\right)\right|^{2}\left(-\rho\right)^{\frac{t}{\eta}}.
\]
By using \eqref{eq:mn3},
\[
\int_{\Omega}\left|P_{\Omega}\left(\delta^{-\frac{t}{2}}\phi\right)\right|^{2}\left(-\rho\right)^{\frac{t}{\eta}}\leq\frac{\eta}{\eta-t}\int_{\Omega}\left|\phi\right|^{2}\delta^{-t}\left(-\rho\right)^{\frac{t}{\eta}}\leq\text{const}_{\Omega,t},
\]
since $-\rho\leq\text{const}_{\Omega}\delta^{\eta}$ and $\left\Vert \phi\right\Vert _{L^{2}\left(\Omega\right)}=1$.
So we obtain that
\begin{equation}
\left|P_{\Omega}\left(\delta^{-\frac{t}{2}}\phi\right)\left(z\right)\right|^{2}\leq\text{const}_{\Omega,t}\,\delta^{-n-1-\frac{t}{\eta}}\left(z\right)\left|\log\delta\left(z\right)\right|^{\frac{t}{c\eta}},\forall z\in\Omega\text{ with }\delta\left(z\right)<\delta_{0}.\label{eq:mn5}
\end{equation}
By the sub-mean inequality, \eqref{eq:mn5} is also
true if $z$ is away from $\partial\Omega$. Indeed, if $\delta_{\Omega}\left(z\right)\geq\delta_{0}$ then
\begin{align*}
\left|P_{\Omega}\left(\delta^{-\frac{t}{2}}\phi\right)\left(z\right)\right|^{2} & \leq\left|B\left(z,\frac{\delta_{0}}{2}\right)\right|^{-1}\int_{B\left(z,\frac{\delta_{0}}{2}\right)}\left|P_{\Omega}\left(\delta^{-\frac{t}{2}}\phi\right)\right|^{2}\\
 & \leq\text{const}_{\Omega,t}\;\int_{B\left(z,\frac{\delta_{0}}{2}\right)}\left|P_{\Omega}\left(\delta^{-\frac{t}{2}}\phi\right)\right|^{2}\left(-\rho\right)^{\frac{t}{\eta}}\\
 & \leq\text{const}_{\Omega,t}.
\end{align*}
From \eqref{combi} and \eqref{eq:mn7}, it
continues as
\[
\left|\int_{\Omega}P_{\Omega}\left(fg\right)\delta^{-\frac{t}{2}}\overline{\phi}\right|\leq\text{const}_{\Omega,t}\int_{\Omega}\left|f\right|\left|\omega_{m,g}\right|\delta^{m-\frac{1}{2}\left(n+1+\frac{t}{\eta}\right)}\left|\log\delta\right|^{\frac{t}{2c\eta}}.
\]
By choosing $m>k+\frac{1}{2}\left(n+1+\left.t\right/\eta\right)$,
we arrive at
\begin{equation}
\left|\int_{\Omega}P_{\Omega}\left(fg\right)\delta^{-\frac{t}{2}}\overline{\phi}\right|\leq\text{const}_{\Omega,g,k,t}\int_{\Omega}\left|f\right|\delta^{k}\leq\text{const}_{\Omega,g,k,t}\int_{\Omega}\left|f\right|\delta_{\Omega}^{k},\label{eq:mn8}
\end{equation}
for any $\phi\in C_{0}^{\infty}\left(\Omega\right)$ so that $\left\Vert \phi\right\Vert _{L^{2}\left(\Omega\right)}=1$.
By an elementary fact given below, from this we get that $P_{\Omega}\left(fg\right)\delta^{-\left.t\right/2}\in L^{2}\left(\Omega\right)$
and 
\[
\left\Vert P_{\Omega}\left(fg\right)\right\Vert _{H^{\left.t\right/2}\left(\Omega\right)}\leq\left(\int_{\Omega}\left|P_{\Omega}\left(fg\right)\right|^{2}\delta^{-t}\right)^{\frac{1}{2}}\leq\text{const}_{\Omega,g,k,t}\int_{\Omega}\left|f\right|\delta_{\Omega}^{k}.
\]
This completes the proof of Theorem \ref{thm1}.

\emph{Fact}: \emph{Let $\varphi$ be a locally integrable function on $\Omega$
such that $\left|\int_{\Omega}\varphi\phi\right|\leq C$, for any
$\phi\in C_{0}^{\infty}\left(\Omega\right)$ with $\left\Vert \phi\right\Vert _{L^{2}\left(\Omega\right)}=1$.
Then $\varphi\in L^{2}\left(\Omega\right)$ and $\left\Vert \varphi\right\Vert _{L^{2}\left(\Omega\right)}\leq C$.}

\emph{Justify the fact}: by the hypothesis, the operator $\phi\rightarrow\int_{\Omega}\varphi\phi$
can be extended boundedly on $L^{2}\left(\Omega\right)$. By Riesz's theorem, there is $\widetilde{\varphi}\in L^{2}\left(\Omega\right)$
such that $\int_{\Omega}\left(\varphi-\widetilde{\varphi}\right)\phi=0$,
for any $\phi\in C_{0}^{\infty}\left(\Omega\right)$. Thus $\varphi=\widetilde{\varphi}$.

\end{proof}

If one extends the domain of definition of $P_{\Omega}$ to the class
\[
\mathcal{F}=\left\{ f\text{ is Lebesgue measurable on }\Omega:f\left(\cdot\right)K\left(z,\cdot\right)\in L^{1}\left(\Omega\right),\text{ a.e. }z\in\Omega\right\} ,
\]
so that now
\[
P_{\Omega}\left(f\right)\left(\cdot\right)=\intop_{\Omega}K\left(\cdot,w\right)f\left(w\right)dV_{w},\quad\quad\text{(for }f\in\mathcal{F}),
\]
is well-defined as a measurable function on $\Omega$, then we obtain
the following estimate:

\begin{proposition}\label{impro}
Let $\Omega$ be a smoothly bounded pseudoconvex domain in $\mathbb{C}^{n}$. Let $\left.-\alpha\left(\Omega\right)\right/2<s_{1}<s_{2}<\left.\alpha\left(\Omega\right)\right/2$. Then for every function $g\in C^{\infty}\left(\overline{\Omega}\right)$ and
any $k\in\mathbb{Z}^{+}$, 
there is a positive constant $C$ such that
\[
\left\Vert P_{\Omega}\left(fg\right)\right\Vert _{H^{s_{2}}\left(\Omega\right)}\leq C \intop_{\Omega}\left|f\right|\delta_{\Omega}^{k},
\]
for all conjugate holomorphic functions $f\in H^{s_{1}}\left(\Omega\right)$.

\end{proposition}

\begin{proof}

Without loss of generality, we may assume that $s_{1}<0<s_{2}$. Then, there are $c\in\left(\max\left\{ -s_{1},s_{2}\right\} ,\left.\alpha\left(\Omega\right)\right/2\right)$
and $\rho\in PSH^{-}\left(\Omega\right)$ such that $-\rho\leq \text{const}\,\delta_{\Omega}^{2c}$
on $\Omega$. Since $A^{2}\left(\Omega\right)=\mathcal{O}\left(\Omega\right)\cap L^{2}\left(\Omega\right)$
is dense in $\mathcal{O}\left(\Omega\right)\cap L^{2}\left(\Omega,\delta_{\Omega}^{-2s_{1}}\right)$
(see the Fact in the proof of Theorem \ref{thm2}), we can choose
a sequence $\left\{ \overline{f_{j}}\right\} $ in $A^{2}\left(\Omega\right)$
that converges to $\overline{f}$ in $L^{2}\left(\Omega,\delta_{\Omega}^{-2s_{1}}\right)$.
It follows that \[
\intop_{\Omega}\left|f_{j}\right|\delta_{\Omega}^{k}\rightarrow\intop_{\Omega}\left|f\right|\delta_{\Omega}^{k},\text{ as }j\rightarrow\infty.
\]
From Theorem \ref{thm1}, there is a positive constant $C$ such that
\[
\left\Vert P_{\Omega}\left(f_{j}g\right)\right\Vert _{H^{s_{2}}\left(\Omega\right)}\leq C \intop_{\Omega}\left|f_{j}\right|\delta_{\Omega}^{k},\;\forall j.
\]
By Fatou's lemma and Weierstrass's theorem, it suffices to
show that $P_{\Omega}\left(f_{j}g\right)$ converges uniformly to
$P_{\Omega}\left(fg\right)$ on any relatively compact subset $W$ of $\Omega$. Take $\widetilde{W}$ so that $W\Subset\widetilde{W}\Subset\Omega$,
then $K_{\widetilde{W}}\left(z,z\right)$ is bounded for $z\in W$. Thus
\begin{align*}\left|\int_{\Omega}K_{\Omega}\left(z,w\right)\left(f_{j}-f\right)\left(w\right)g\left(w\right)dV_{w}\right|^{2}\\
 &\hspace*{-2.3cm} \leq\left(\int_{\Omega}\left|K_{\Omega}\left(z,\cdot\right)\right|^{2}\left(-\rho\right)^{\frac{s_{1}}{c}}\right)\left(\int_{\Omega}\left|f_{j}-f\right|^{2}\left|g\right|^{2}\left(-\rho\right)^{-\frac{s_{1}}{c}}\right)\\
 &\hspace*{-2.3cm} \leq\;\text{c}_{\widetilde{W},g} \;K_{\widetilde{W} }\left(z,z\right)\left\Vert f_{j}-f\right\Vert _{L^{2}\left(\Omega,\delta_{\Omega}^{-2s_{1}}\right)}^{2}\\
 &\hspace*{-2.3cm} \rightarrow0,
\end{align*}
uniformly in $z\in W$, as $j\rightarrow\infty$.

\end{proof}

\section{Proof of Theorem \ref{thm2}}

We first verify the following weighted estimate, which follows the
proof of Theorem \ref{BeCh1} very closely. However, here we need
to employ a slightly different limiting argument.

\begin{lemma}\label{lewe}
Let $\Omega$ be a smoothly bounded pseudoconvex domain in $\mathbb{C}^{n}$.
Let $\delta\in C^{\infty}\left(\overline{\Omega}\right)$ be a function
that equals $\delta_{\Omega}$ near $\partial\Omega$. Let $\rho\in PSH^{-}\left(\Omega\right)$
be a function in the form $\rho=-\delta^{\eta}\Psi$, where $0<\eta<1$,
and $\Psi\in C^{\infty}\left(\overline{\Omega}\right)$ such that
$\Psi>0$ on $\overline{\Omega}$.  Let $\beta\geq0$, $\alpha\in\mathbb{R}$ be numbers such that
$\beta^{2}<\alpha+\beta$. Let $\varphi=-\alpha\log\left(-\rho\right)$
and $\psi=-\beta\log\left(-\rho\right)$. Then
\begin{equation}\label{bala_1}
\intop_{\Omega}\left|P_{\varphi}\left(f\right)\right|^{2}\left(1-\frac{\beta^{2}}{\alpha+\beta}\right)e^{\psi-\varphi}\leq\int_{\Omega}\left|P_{\varphi+\psi}\left(e^{\psi}f\right)\right|^{2}e^{-\left(\varphi+\psi\right)},
\end{equation}
for any function $f\in L^{2}\left(\Omega,e^{\psi-\varphi}\right)$.
\end{lemma}

\begin{proof}
Observe that $h=e^{-\psi}P_{\varphi+\psi}\left(e^{\psi}f\right)\in L^{2}\left(\Omega,e^{-\varphi}\right)$,
since $\beta\geq0$ and $\rho\in L^{\infty}\left(\Omega\right)$.
Also $he^{\psi}\in L^{2}\left(\Omega,e^{-\varphi-\psi}\right)$, since
$f\in L^{2}\left(\Omega,e^{\psi-\varphi}\right)$. From Theorem \ref{BeCh1},
\begin{align*}
\int_{\Omega}\left|P_{\varphi}\left(h\right)\right|^{2}\left(1-\frac{\beta^{2}}{\alpha+\beta}\right)e^{\psi-\varphi} & \leq\int_{\Omega}\left|h\right|^{2}e^{\psi-\varphi}\\
 & \leq\int_{\Omega}\left|P_{\varphi+\psi}\left(e^{\psi}f\right)\right|^{2}e^{-\left(\varphi+\psi\right)}\\
 & \leq\int_{\Omega}\left|f\right|^{2}e^{\psi-\varphi}\\
 & <\infty.
\end{align*}
Thus $u_{\varphi}e^{\psi}\in L^{2}\left(\Omega,e^{-\varphi-\psi}\right)$,
where $u_{\varphi}=h-P_{\varphi}\left(h\right)$. To show that $u_{\varphi}e^{\psi}$
is the $L^{2}\left(\Omega,e^{-\varphi-\psi}\right)$-minimal solution
to the dee-bar equation $\overline{\partial}u=\overline{\partial}\left(u_{\varphi}e^{\psi}\right)$,
we need to verify that for any $v\in A^{2}\left(\Omega,e^{-\varphi-\psi}\right)=L^{2}\left(\Omega,e^{-\varphi-\psi}\right)\cap\mathcal{O}\left(\Omega\right)$,
\begin{equation}
\intop_{\Omega}u_{\varphi}\overline{v}e^{-\varphi}=0.\label{eq:lim1}
\end{equation}
From $\beta^{2}<\alpha+\beta$, we get that $\alpha>\left.-1\right/4$.
Thus $-\left(\alpha+\beta\right)\eta\leq-\alpha\eta<1$. Therefore,
by using  the Fact in the proof of Theorem \ref{thm2}, $A^{2}\left(\Omega,e^{-\varphi}\right)=L^{2}\left(\Omega,\delta_{\Omega}^{\alpha\eta}\right)\cap\mathcal{O}\left(\Omega\right)$
is dense in $A^{2}\left(\Omega,e^{-\varphi-\psi}\right)=L^{2}\left(\Omega,\delta_{\Omega}^{\left(\alpha+\beta\right)\eta}\right)\cap\mathcal{O}\left(\Omega\right)$.
So, there is a sequence $\left\{ h_{j}\right\} \in A^{2}\left(\Omega,e^{-\varphi}\right)$
converging to $v$ in $L^{2}\left(\Omega,e^{-\varphi-\psi}\right)$.
It is clear that 
\[
\int_{\Omega}u_{\varphi}\overline{h_{j}}e^{-\varphi}=0,\;\forall j.
\]
On the other hand,
\begin{align*}
\left|\int_{\Omega}u_{\varphi}\left(\overline{h_{j}}-\overline{v}\right)e^{-\varphi}\right| & \leq\left(\int_{\Omega}\left|u_{\varphi}\right|^{2}e^{\psi-\varphi}\right)^{\frac{1}{2}}\left(\int_{\Omega}\left|h_{j}-v\right|^{2}e^{-\left(\varphi+\psi\right)}\right)^{\frac{1}{2}}\\
 & \leq\text{const}\;\left(\int_{\Omega}\left|h_{j}-v\right|^{2}e^{-\left(\varphi+\psi\right)}\right)^{\frac{1}{2}}\\
 & \rightarrow0,\text{ as }j\rightarrow\infty.
\end{align*}
Thus \eqref{eq:lim1} follows. The rest of the argument is the same
as in  the proof of Theorem \ref{BeCh1}. That is, we now have 
\[
\int_{\Omega}\left|u_{\varphi}\right|^{2}e^{\psi-\varphi}\leq\frac{\beta^{2}}{\alpha+\beta}\int_{\Omega}\left|P_{\varphi}\left(h\right)\right|^{2}e^{\psi-\varphi},
\]
and also
\begin{align*}
\int_{\Omega}\left|P_{\varphi}\left(f\right)\right|^{2}e^{\psi-\varphi} & =\int_{\Omega}\left|P_{\varphi}\left(h\right)\right|^{2}e^{\psi-\varphi}\\
 & =\int_{\Omega}\left|u_{\varphi}\right|^{2}e^{\psi-\varphi}+\int_{\Omega}\left|h\right|^{2}e^{\psi-\varphi}\\
 & =\int_{\Omega}\left|u_{\varphi}\right|^{2}e^{\psi-\varphi}+\int_{\Omega}\left|P_{\varphi+\psi}\left(e^{\psi}f\right)\right|^{2}e^{-\left(\varphi+\psi\right)}.
\end{align*}
The desired conclusion then follows.

\end{proof}

\begin{proof}[Proof of Theorem \ref{thm2}]
Our approach is basically similar to the proof of Theorem \ref{thm1}.
An observation is that the estimate \eqref{bala_1} in Lemma \ref{lewe}
is useful to give balanced estimates in weights.

From the ranges of $s_{1}$, $s_{2}$ and $\alpha$, we can choose
$\beta>0$ such that $\beta^{2}<\alpha+\beta$ and 
\[
-\left(\beta+\alpha\right)\frac{\eta}{2}<s_{1}<s_{2}<\left(\beta-\alpha\right)\frac{\eta}{2}.
\]
We now apply $\varphi=-\alpha\log\left(-\rho\right)$ and $\psi=-\beta\log\left(-\rho\right)$
to \eqref{bala_1}. Let's also assume for a moment that $\overline{f}\in L^{2}\left(\Omega,e^{\psi-\varphi}\right)\,\cap\,\mathcal{O}\left(\Omega\right)$. Lemma
\ref{lewe} gives 
\begin{align}
\left\Vert P_{\left(-\rho\right)^{\alpha}}\left(fg\right)\right\Vert _{H^{\left(\beta-\alpha\right)\left(\left.\eta\right/2\right)}\left(\Omega\right)}^{2} & \leq\int_{\Omega}\left|P_{\left(-\rho\right)^{\alpha}}\left(fg\right)\right|^{2}\delta^{-\left(\beta-\alpha\right)\eta}\nonumber \\
 & \leq C_{1}\int_{\Omega}\left|P_{\varphi+\psi}\left(e^{\psi}fg\right)\right|^{2}e^{-\left(\varphi+\psi\right)},\label{eq:mnr2}
\end{align}
where $\left.1\right/C_{1}=\left(1-\left.\beta^{2}\right/\left(\alpha+\beta\right)\right)\inf_{\overline{\Omega}}\Psi^{-\left(\beta-\alpha\right)}$. Take any $\phi\in C_{0}^{\infty}\left(\Omega\right)$ such that $\left\Vert \phi\right\Vert _{L^{2}\left(\Omega\right)}=1$. By duality,
\begin{align*} \left|\int_{\Omega}P_{\varphi+\psi}\left(e^{\psi}fg\right)\overline{\phi}e^{-\left.\left(\varphi+\psi\right)\right/2}\right| & =\left|\int_{\Omega}P_{\varphi+\psi}\left(e^{\psi}fg\right)\overline{\phi}e^{\left.\left(\varphi+\psi\right)\right/2}e^{-\left(\varphi+\psi\right)}\right|\\  & =\left|\int_{\Omega}e^{\psi}fg\overline{P_{\varphi+\psi}\left(\phi e^{\left.\left(\varphi+\psi\right)\right/2}\right)}e^{-\left(\varphi+\psi\right)}\right|\\  & =\left|\int_{\Omega}\overline{f}P_{\varphi+\psi}\left(\phi e^{\left.\left(\varphi+\psi\right)\right/2}\right)\overline{g}e^{-\varphi}\right|, 
\end{align*}
here we have used the fact that $\phi e^{\left.\left(\varphi+\psi\right)\right/2}\in L^{2}\left(\Omega,e^{-\left(\varphi+\psi\right)}\right)$. Notice that $\overline{g}e^{-\varphi}=\delta^{\eta\alpha}\overline{g}\Psi^{\alpha}$, and $\overline{g}\Psi^{\alpha}\in C^{\infty}\left(\overline{\Omega}\right)$. Also $\overline{f}P_{\varphi+\psi}\left(\phi e^{\left.\left(\varphi+\psi\right)\right/2}\right)\in\mathcal{O}\left(\Omega\right)$ and
\begin{align*}\left|\int_{\Omega}\overline{f}P_{\varphi+\psi}\left(\phi e^{\left.\left(\varphi+\psi\right)\right/2}\right)\delta^{\eta\alpha}\right|\\  
& \hspace*{-2.0cm}=\left|\int_{\Omega}\overline{f}e^{\left.\left(\psi-\varphi\right)\right/2}\Psi^{-\alpha}P_{\varphi+\psi}\left(\phi e^{\left.\left(\varphi+\psi\right)\right/2}\right)e^{\left.-\left(\psi+\varphi\right)\right/2}\right|\\  
& \hspace*{-2.0cm}\leq C_{2}\left(\int_{\Omega}\left|f\right|^{2}e^{\psi-\varphi}\right)^{\frac{1}{2}}\left(\int_{\Omega}\left|P_{\varphi+\psi}\left(\phi e^{\left.\left(\varphi+\psi\right)\right/2}\right)\right|^{2}e^{-\left(\psi+\varphi\right)}\right)^{\frac{1}{2}}\\  
& \hspace*{-2.0cm}\leq C_{2}\left(\int_{\Omega}\left|f\right|^{2}e^{\psi-\varphi}\right)^{\frac{1}{2}}, 
\end{align*}
where $C_{2}=\left\Vert \Psi^{-\alpha}\right\Vert _{L^{\infty}\left(\Omega\right)}$. Thus $\overline{f}P_{\varphi+\psi}\left(\phi e^{\left.\left(\varphi+\psi\right)\right/2}\right)\delta^{\eta\alpha}\in L^{1}\left(\Omega\right)$. Now, Theorem \ref{lem2} can be used to give 
\begin{equation} \int_{\Omega}\overline{f}P_{\varphi+\psi}\left(\phi e^{\left.\left(\varphi+\psi\right)\right/2}\right)\overline{g}e^{-\varphi}=\int_{\Omega}\overline{f}P_{\varphi+\psi}\left(\phi e^{\left.\left(\varphi+\psi\right)\right/2}\right)\omega_{m}\delta^{m},\label{eq:mnr3} 
\end{equation} 
where $\omega_{m}\in L^{\infty}\left(\Omega\right)$, for $m\in\mathbb{Z}^{+}$ to be specified later. On the other hand
\begin{align*}\left|P_{\varphi+\psi}\left(\phi e^{\left.\left(\varphi+\psi\right)\right/2}\right)\left(z\right)\right|^{2}\\
 &\hspace*{-2.4cm} \leq\text{const}_{n}\,K\left(z,z\right)\int_{\left\{ g_{\Omega}\left(\cdot,z\right)<-1\right\} }\left|P_{\varphi+\psi}\left(\phi e^{\left.\left(\varphi+\psi\right)\right/2}\right)\right|^{2}\\
 & \hspace*{-2.4cm}\leq\text{c}_{\Omega}\,\frac{\delta^{-n-1-\eta\left(\alpha+\beta\right)}\left(z\right)}{\left|\log\delta\left(z\right)\right|^{\left.-\eta\left(\alpha+\beta\right)\right/c}}\int_{\left\{ g_{\Omega}\left(\cdot,z\right)<-1\right\} }\left|P_{\varphi+\psi}\left(\phi e^{\left.\left(\varphi+\psi\right)\right/2}\right)\right|^{2}e^{-\left(\varphi+\psi\right)}\\
 & \hspace*{-2.4cm}\leq\text{c}_{\Omega}\,\frac{\delta^{-n-1-\eta\left(\alpha+\beta\right)}\left(z\right)}{\left|\log\delta\left(z\right)\right|^{\left.-\eta\left(\alpha+\beta\right)\right/c}},\;\forall z\in\Omega.
\end{align*}
Combining this with \eqref{eq:mnr2}, \eqref{eq:mnr3} and the elementary fact in the proof of Theorem \ref{thm1}, we obtain that  
\begin{equation}
\int_{\Omega}\left|P_{\left(-\rho\right)^{\alpha}}\left(fg\right)\right|^{2}\delta^{-\left(\beta-\alpha\right)\eta}\leq\text{const}\left(\int_{\Omega}\left|f\right|\delta^{k}\right)^{2},\label{eq:mnr4}
\end{equation}
for any conjugate holomorphic function $f\in L^{2}\left(\Omega,e^{\psi-\varphi}\right)$,
provided that $$m>k+\left(\left.1\right/2\right)\left(n+1+\eta\left(\alpha+\beta\right)\right).$$
We now show that \eqref{eq:mnr4} is also true for all $\overline{f}\in H^{s_{1}}\left(\Omega\right)\,\cap\,\mathcal{O}\left(\Omega\right)$.
Fix any $\overline{f}\in H^{s_{1}}\left(\Omega\right)\cap\mathcal{O}\left(\Omega\right)$.
Since $-\eta\left(\beta+\alpha\right)<\eta\left(\beta-\alpha\right)<1$,
by the fact below, there is a sequence $\left\{ \overline{f_{j}}\right\} \in L^{2}\left(\Omega,\delta^{-\eta\left(\beta-\alpha\right)}\right)$
that converges to $\overline{f}$ in $L^{2}\left(\Omega,\delta^{\eta\left(\beta+\alpha\right)}\right)$.
Since
\[
\int_{\Omega}\left|f_{j}\right|\delta^{k}\rightarrow\int_{\Omega}\left|f\right|\delta^{k}\text{ as }j\rightarrow\infty,
\]
and $\overline{f_{j}}\in L^{2}\left(\Omega,e^{\psi-\varphi}\right)\cap\mathcal{O}\left(\Omega\right)$,
thus all we need to do is show that $P_{\left(-\rho\right)^{\alpha}}\left(f_{j}g\right)$
converges uniformly to $P_{\left(-\rho\right)^{\alpha}}\left(fg\right)$
on any subset $W\Subset\Omega$. Take a subset $\widetilde{W}\Subset\Omega$
that contains $W$, by H{\"o}lder's inequality
\begin{align*}
\left|\int_{\Omega}K_{\left(-\rho\right)^{\alpha}}\left(z,w\right)\left(f_{j}-f\right)\left(w\right)g\left(w\right)\left(-\rho\right)^{\alpha}\left(w\right)dV_{w}\right|^{2}\\
 &\hspace*{-7cm} \leq\text{const}\,\left(\int_{\Omega}\left|K_{\left(-\rho\right)^{\alpha}}\left(z,w\right)\right|^{2}\left(-\rho\right)^{-\left(\beta-\alpha\right)}\left(w\right)dV_{w}\right)\left\Vert f_{j}-f\right\Vert _{L^{2}\left(\Omega,\delta^{\eta\left(\beta+\alpha\right)}\right)}^{2}\\
 & \hspace*{-7cm}\leq\text{const}\,K_{\widetilde{W},\left(-\rho\right)^{\alpha}}\left(z,z\right)\left\Vert f_{j}-f\right\Vert _{L^{2}\left(\Omega,\delta^{\eta\left(\beta+\alpha\right)}\right)}^{2}\\
 & \hspace*{-7cm}\rightarrow 0,
\end{align*}
uniformly in $z\in W$, as $j\rightarrow \infty$. Given the fact below, this completes the proof of Theorem \ref{thm2}. 

\vspace*{0.2cm}

\emph{Fact}: \emph{Let $\Omega$ be a smoothly, bounded pseudoconvex domain in $\mathbb{C}^{n}$,
and let $$-\infty<\alpha_{1}<\alpha_{2}<1.$$ Then $A_{\alpha_{2}}^{2}\left(\Omega\right)=\mathcal{O}\left(\Omega\right)\cap L^{2}\left(\Omega,\delta_{\text{\ensuremath{\Omega}}}^{-\alpha_{2}}\right)$
is dense in $A_{\alpha_{1}}^{2}\left(\Omega\right)=\mathcal{O}\left(\Omega\right)\cap L^{2}\left(\Omega,\delta_{\text{\ensuremath{\Omega}}}^{-\alpha_{1}}\right)$.}

\emph{Justify the fact.} 
If $\alpha_{1}>0$ then the claim is proved in \cite[page 4135]{Che14}, which in fact only requires $C^{2}$ smoothness of the boundary. When $\alpha_{1}\leq 0$, it suffices to show that 
$A^{\infty}\left(\Omega\right)=\mathcal{O}\left(\Omega\right)\cap C^{\infty}\left(\overline{\Omega}\right)$
 is dense in $A_{\alpha_{1}}^{2}\left(\Omega\right)$. Since $A^{\infty}\left(\Omega\right)$ is dense in $A^{2}\left(\Omega\right)=\mathcal{O}\left(\Omega\right)\cap L^{2}\left(\Omega\right)$,
see \cite[Theorem 3.1.4]{Cat80}. Thus it suffices to show that $A^{2}\left(\Omega\right)$
is dense in $A_{\alpha_{1}}^{2}\left(\Omega\right)$. Indeed, if this is true then for any $f\in A_{\alpha_{1}}^{2}\left(\Omega\right)$ and $\varepsilon>0$, one can choose $g\in A^{2}\left(\Omega\right)$
and $h\in A^{\infty}\left(\Omega\right)$ such that
\[
\left\Vert f-g\right\Vert _{L^{2}\left(\Omega,\delta_{\Omega}^{-\alpha_{1}}\right)}<\frac{\varepsilon}{2},\text{ and }\left\Vert g-h\right\Vert _{L^{2}\left(\Omega,\delta_{\Omega}^{-\alpha_{1}}\right)}\leq C_{\Omega}\left\Vert g-h\right\Vert _{L^{2}\left(\Omega\right)}<\frac{\varepsilon}{2}.
\]
So $\left\Vert f-h\right\Vert _{L^{2}\left(\Omega,\delta_{\Omega}^{-\alpha_{1}}\right)}<\varepsilon$. The remaining task can also be proved using the dee-bar technique as in 
\cite{Che14}. For convenience, we include it here. 

Fix any $f\in A_{\alpha_{1}}^{2}\left(\Omega\right)$, and choose a function $\rho\in C^{2}\left(\Omega\right)\cap PSH^{-}\left(\Omega\right)$
such that $\left(\left.1\right/C\right)\delta_{\Omega}^{c}\leq-\rho\leq C\delta_{\Omega}^{c}$,
for some positive constants $C$ and $c$. Let $\chi\in C^{\infty}\left(\mathbb{R}\right)$
be such that $\left.\chi\right|_{\left(0,\infty\right)}=0$ and $\left.\chi\right|_{\left(-\infty,-\log2\right)}=1$.
For each $\varepsilon>0$, let $\varphi_{\varepsilon}\in PSH\left(\Omega\right)$
be defined by $\varphi_{\varepsilon}=-\left(\left.-\alpha_{1}\right/c\right)\log\left(-\rho+\varepsilon\right)$.
Apply $L^{2}$-H{\"o}rmander estimate to the dee-bar equation 
\[
\overline{\partial}u=f\overline{\partial}\chi\left(-\log\left(-\rho+\varepsilon\right)+\log2\varepsilon\right),
\]
then we can find a solution $u_{\varepsilon}$ satisfying the estimates
\begin{align*}
\int_{\Omega}\left|u_{\varepsilon}\right|^{2}e^{-\varphi_{\varepsilon}} & \leq\int_{\Omega}\left|f\overline{\partial}\chi\left(-\log\left(-\rho+\varepsilon\right)+\log2\varepsilon\right)\right|_{i\partial\overline{\partial}\varphi_{\varepsilon}}^{2}e^{-\varphi_{\varepsilon}}\\
 & \leq\text{const}\int_{\varepsilon\leq-\rho\leq3\varepsilon}\left|f\right|^{2}\delta_{\Omega}^{-\alpha_{1}}.
\end{align*}
Since $e^{-\varphi_{\varepsilon}}>\varepsilon^{\left.-\alpha_{1}\right/c}$,
thus $u_{\varepsilon}\in L^{2}\left(\Omega\right)$. Also 
\[
\int_{\Omega}\left|f\chi\left(-\log\left(-\rho+\varepsilon\right)+\log2\varepsilon\right)\right|^{2}\leq\text{const}\int_{-\rho\geq\varepsilon}\left|f\right|^{2}\leq\text{const}\,\varepsilon^{\left.\alpha_{1}\right/c}\int_{\Omega}\left|f\right|^{2}\delta_{\Omega}^{-\alpha_{1}}.
\]
Therefore $f_{\varepsilon}:=f\chi\left(-\log\left(-\rho+\varepsilon\right)+\log2\varepsilon\right)-u_{\varepsilon}\in A^{2}\left(\Omega\right)$,
and we also have
\begin{align*}
\int_{\Omega}\left|f_{\varepsilon}-f\right|^{2}\delta_{\Omega}^{-\alpha_{1}} & \leq\text{const}\left(\int_{-\rho\leq3\varepsilon}\left|f\right|^{2}\delta_{\Omega}^{-\alpha_{1}}+\int_{\Omega}\left|u_{\varepsilon}\right|^{2}\delta_{\Omega}^{-\alpha_{1}}\right)\\
 & \leq\text{const}\left(\int_{-\rho\leq3\varepsilon}\left|f\right|^{2}\delta_{\Omega}^{-\alpha_{1}}+\int_{\Omega}\left|u_{\varepsilon}\right|^{2}e^{-\varphi_{\varepsilon}}\right)\\
 & \leq\text{const}\left(\int_{-\rho\leq3\varepsilon}\left|f\right|^{2}\delta_{\Omega}^{-\alpha_{1}}+\int_{\varepsilon\leq-\rho\leq3\varepsilon}\left|f\right|^{2}\delta_{\Omega}^{-\alpha_{1}}\right)\\
 & \rightarrow0\text{ as }\varepsilon\rightarrow0^{+}.
\end{align*}
So we have verified that $A^{2}\left(\Omega\right)$ is dense in $A_{\alpha_{1}}^{2}\left(\Omega\right)$.

\end{proof}

\textbf{Acknowledgement.} The author would like to thank Professor Bo-Yong Chen for his stimulating research on the topic.

\end{document}